\documentclass[12pt,twoside ]{article}
\usepackage{graphicx}
\usepackage[french,english]{babel}
\usepackage[utf8]{inputenc} 
\usepackage[T1]{fontenc} 
\usepackage{array} 
\usepackage{tabularx,booktabs,multirow} 
\usepackage{amsmath, amsfonts, amssymb, amsthm, mathtools, thmtools, mathrsfs} 
\usepackage{ dsfont } 
\usepackage{lmodern} 
\usepackage{ytableau, youngtab, genyoungtabtikz} 
\ytableausetup{aligntableaux=center}
\usepackage{cancel} 
\usepackage{comment}
\usepackage{url}
\usepackage{caption}
\usepackage{subcaption}
\usepackage[normalem]{ulem}
\usepackage{makecell}
\usepackage{rotating}
\usepackage{longtable}
\usepackage{fullpage}
\usepackage{pdflscape}
\usepackage{xcolor}
\usepackage{enumitem}
\usepackage{scalerel}

\usepackage{imakeidx}
\makeindex[title=Index des définitions, program=makeindex, columns=2]
\let\emphorig\emph
\renewcommand{\emph}[1]{\emphorig{#1}\index{#1}}

\newtheoremstyle{mes_theoremes}{1.5em}{2em}{}{}{\bfseries}{~:~}{\parskip}{\thmname{#1}\thmnumber{ #2}\thmnote{ (#3)}}
\theoremstyle{mes_theoremes} 
\newtheorem{theo}{Theorem}[section]
\newtheorem{theo*}{Theorem}
\newtheorem{prop}[theo]{Proposition}
\newtheorem{conj}[theo]{Conjecture}
\newtheorem{cor}[theo]{Corollary}

\newtheorem{algo}[theo]{Algorithm}

\newtheorem{rem}[theo]{Remark}

\newtheorem*{ex*}{Example}
\newtheorem{ex}[theo]{Example}

\usepackage[colorinlistoftodos]{todonotes}
\newcommand{\flo}[1]{\todo[size=\tiny,color=green!30,inline]{#1}}

\newcommand{\N}{\mathbb{N}}

\newcommand{\C}{\mathbb{C}}

\newcommand{\isom}{\simeq}

\renewcommand{\S}{\mathds{S}}
\renewcommand{\O}{\mathcal{O}}

\Yboxdim{0.4cm} 
\newcommand\rd{\Yfillcolour{red!30}}
\newcommand\bl{\Yfillcolour{blue!30}}

\newcommand\wt{\Yfillcolour{white}}
\newcommand\red{\Ynodecolour{red}}

\newcommand\black{\Ynodecolour{black}}

\usepackage{tikz, tikz-3dplot, pgfplots, tikzit}
\usepackage{tkz-graph}
\usepackage{ulem}
\usetikzlibrary[positioning,patterns,shapes, calc, arrows,decorations.markings]

\pgfdeclareshape{3rRevL}
{
	\savedanchor\centerpoint{%
		\pgf@x=.5\wd\pgfnodeparttextbox%
		\pgf@y=.5\ht\pgfnodeparttextbox%
		\advance\pgf@y by -.5\dp\pgfnodeparttextbox%
	}
	\anchor{center}{\centerpoint}
	\anchorborder{\centerpoint}
	\anchor{text}{%
		\centerpoint%
		\pgf@x=-\pgf@x%
		\pgf@y=-\pgf@y%
	}
}

\pgfplotsset{compat=1.8}
\tikzstyle{bsq}=[rectangle, draw, thick, minimum width=1cm, minimum height=1cm] 
\tikzstyle{3rver}=[rectangle, draw, thick, minimum width=1cm, minimum height=3cm]
\tikzstyle{3rhor}=[rectangle, draw, thick, minimum width=3cm, minimum height=1cm]
\def\3rRevL at (#1,#2){\draw (#1,#2) -- ++(0,2) -- ++(2,0) -- ++(0,-1) -- ++(-1,0) -- ++(0,-1) -- cycle;}
\def\3rJ at (#1,#2){\draw (#1,#2) -- ++(0,1) -- ++(1,0) -- ++(0,1) -- ++(1,0) -- ++(0,-2) -- cycle;}

\begin{document}
	
\title{
Quasicrystal Structure of Fundamental Quasisymmetric Functions, and Skeleton of Crystals}

	
\author{Florence Maas-Gariépy}
	
\thispagestyle{empty}        
\maketitle
	

\begin{abstract}
We use crystals of tableaux and descent compositions to understand the decomposition of Schur functions $s_\lambda$ into Gessel's fundamental quasisymmetric functions $F_\alpha$. The connected crystal of tableaux $B(\lambda)$, associated to $s_\lambda$, is shown to be partitionned into a disjoint union of connected induced subgraphs $B(T_\alpha)$ corresponding to the $F_\alpha$'s.

We show that these subgraphs, which we call quasicrystals, are isomorphic (as graphs) to specific crystals of tableaux. This allows us to give a formula for the number of tableaux of shape $\lambda$ and maximal entry $n$. We also use this setting to give a constructive proof of a combinatorial formula for Kostka numbers $K^\lambda_\mu$. We study the position of the quasicrystals within the crystal $B(\lambda)$, and show that they appear in dually positionned pairs, with the crystal anti-automorphism between them being given by a generalization of Schützenberger's evacuation. We introduce the notion of skeleton of the crystal $B(\lambda)$ given by replacing each subgraph $B(T_\alpha)$ by the associated standard tableau of shape $\lambda$. We conjecture that its graph includes the dual equivalence graph for $\lambda$, introduced by Assaf, and that its subgraphs of tableaux with fixed number of descents have particular structures.
Finally, we describe applications to plethysm, among which we give an algorithm to express any symmetric sum of fundamental quasisymmetric functions into the Schur basis, whose construction gives insight into the relationship between the two basis.
\end{abstract}

\section*{Introduction}

Quasisymmetric functions were introduced by Gessel \cite{Gessel}, 
in the context of the study of symmetric functions, which they generalise. 
Notably, the plethysm $s_\mu[s_\lambda]$ of two Schur functions has been shown to be a sum of fundamental quasisymmetric functions \cite{LoehrWarrington}. Understanding the decomposition of such a plethysm into the Schur basis has been an open problem for more than 80 years, 
since its initial introduction by Littlewood \cite{Littlewood}. 
Therefore, studying both the basis of fundamentatal quasisymmetric functions and Schur functions, 
and relations
between them, has the potential of greatly advancing our understanding of plethysm.\\

We propose here a study of Schur functions and fundamental quasisymmetric functions through crystal theory. We believe this to be of value, since both plethysm and crystal theory originate from representation theory.\\ 

A \emph{crystal} 
is a visual representation of the character of a representation of a group 
in the shape of a labelled oriented graph on combinatorial objects.
Irreducible characters then correspond to connected components of crystals, which then are crystals in their own right.
\\
%
%
A good introductory reference on crystals is 
\textit{Crystals for dummies} \cite{CrystalsForDummies}. For a thorough understanding, see \textit{Crystal Bases} \cite{BumpSchilling}.\\

We will focus on crystals of type $A_{n-1}$, which correspond to characters of representations of $GL_n$, which themselves 
are given by symmetric functions: formal power series such that permuting any two variables gives the same function. The Schur functions mentionned above encode the irreducible characters for $GL_n$. 
The problem of decomposing a symmetric function into the basis of Schur functions then corresponds to breaking down a character (and the associated representation) into its smallest pieces: irreducible characters (or representations). In the setting of crystals, we are interested in understanding the decomposition of crystals into connected components $B(\lambda)_n$, which then correspond to the Schur functions $s_\lambda(x_1,\ldots,x_n)$.\\

Furthermore, the following formula of Gessel \cite{Gessel2019} tells us that these connected components can be decomposed further into subcomponents which then correspond to fundamental quasisymmetric functions.
\[
s_\lambda = \sum_{T\in \text{SYT}(\lambda)} F_{DesComp(T)}
\]

Understanding this decomposition in the crystal setting is the first aim of the article. The second aim is to understand the added structure on quasisymmetric functions within the crystal structure, and relations between them. We show the following.


\begin{theo*}
	The connected crystal $B(\lambda)_n$ of tableaux of shape $\lambda$ and maximal entry $n$ is partitionned into disjoint connected induced subgraphs $B(T_\alpha)_n$ which correspond to quasisymmetric functions $F_\alpha(x_1,\ldots,x_n)$, where the subsets of vertices are tableaux with a fixed descent composition $\alpha$. The sources $T_\alpha$ of these subcomponents have filling and \textit{minimal parsing of type $\alpha$}. 
	The number $f^\lambda_\alpha$ of subgraphs of type $\alpha$ is the number of standard tableaux of shape $\lambda$ and descent composition $\alpha$.
\end{theo*}

As a consequence of the following theorem, we have that for a fixed composition $\alpha$, all subcomponents $B(T_\alpha)_n$ are isomorphic as labelled oriented graphs (see corollary~\ref	{prop:GraphStructureSubcomponents}). Moreover this is true no matter the crystal $B(\lambda)_n$ they live in, so no matter the shapes of the sources $T_\alpha$. We denote the class of such subcomponents $B(\alpha)_n$, and call them \textit{quasicrystals}. We can then study their graph structure: 

\begin{theo*}
	Let $\alpha$ be a composition of $m$ in $s$ parts. The quasicrystal $B(\alpha)_n$ is isomorphic (as an oriented graph) to $B(m)$ with maximal entry $n-s+1$. In particular, the oriented graph structure of $B(\alpha)_n$ is independant of partitions $\lambda$ for which $\alpha$ is a descent composition.
	\label{theo:IsomWithBm}
\end{theo*}

An application of both theorems above, and of proposition~\ref{prop:NumberOfTableauxInBm}, is the following formula, giving the number of tableaux with fixed shape $\lambda$ and maximal entry $n$.

\begin{theo*}
	The number of tableaux of shape $\lambda$ with maximal entry $n$ is given by 
	\[|SSYT(\lambda)_n|=
	\sum_{0\leq d\leq D}
	f^\lambda_d\cdot
	\left(\begin{array}{c} |\lambda|+n-d-1 \\ n-d-1 \end{array}\right),
	\]
	where $f^\lambda_d$ denotes the number of standard tableaux of shape $\lambda$ with $d$ descents, and $D$ is the maximal number of descents in a standard tableau of shape $\lambda$.
\end{theo*}

The existence of "nice" formulas for counting tableaux has been an open question for many years. For the similar problem of giving a "nice" formula for Kostka numbers $K^\lambda_\mu$, open for many years, it has been conjectured that no "nice" formula exists, given the chaotic behavior of those numbers \cite{Stanley2} (see the discussion on mathoverflow on this question \cite{mathoverflow}, viewed a thousand times and with contributions from experts). 

There is however a known combinatorial formula below, using descents, which had only a bijective proof. 
The setting of quasicrystals allowed us to give it a constructive proof. 

\begin{prop}[Proposition~\ref{prop:SaganKostka} {\cite[Proposition 5.3.6]{Sagan}} ]
	The Kostka number $K^\lambda_\mu$ counting tableaux of shape $\lambda$ and composition weight $\mu$ is given by the following formula, where $\alpha\preccurlyeq \mu$ if $\mu$ is a refinement of $\alpha$.
	\[
	K^\lambda_\mu = \left|\{T\in SYT(\lambda) \ | \ DesComp(T)\preccurlyeq \mu \} \right|
	\]
\label{prop:SaganKostka}
\end{prop}

After studying the structure of the quasicrystals $B(\alpha)$, we go back to studying the structure of the crystal $B(\lambda)$.

\begin{theo*}
	Subcomponents $B(T_\alpha)$ and $B(T_{\overset{\leftarrow}{\alpha}})$ necessarily appear in pairs in a given crystal $B(\lambda)$. They are then dual one to another (as graphs) and are positionned in dual locations within $B(\lambda)$, with the anti-automorphism of crystal between both being the evacuation map.
\end{theo*}

By replacing each subcomponent $B(T_\alpha)$ in $B(\lambda)_n$ by the associated standard tableau, and preserving only the edges of minimal index between subcomponents, 
we obtain what we call the skeleton of $B(\lambda)_n$, denoted $Skeleton(\lambda)_n$.

\begin{theo*}
	For $\lambda\vdash m$ fixed, let $S$ be the maximal length of a descent composition for tableaux of shape $\lambda$. Then the skeletons $Skeleton(\lambda)_n$ of the crystals $B(\lambda)_n$ are equal for all $n\geq S$.
	For $1\leq n\leq S$, the skeleton of $B(\lambda)_n$ is the induced subgraph of $Skeleton(\lambda)_S$ containing standard tableaux of shape $\lambda$ with descent composition having at most $n$ parts. 
\end{theo*}

The skeleton is then determined for all $n$, and we can define $Skeleton(\lambda) = Skeleton(\lambda)_S$ 
to be the skeleton of $B(\lambda)$, thus giving its underlying structure.\\

We conjecture that the (unoriented and unlabelled) graph structure of the skeleton contains that of the dual equivalence graph for $\lambda$, introduced by Assaf \cite{AssafDualEquivGraph} (see conjecture~\ref{conj:dualEquivGraph}). 
We also conjecture that the induced subgraph of the skeleton holding standard tableaux with fixed number of descents have interesting structures (see conjecture~\ref{conj:structureSubGraphsSkeleton}). \\

\setcounter{theo*}{0}

The third aim of the article is to give some applications to plethysm of the results above (see section~\ref{section:plethysm}), notably counting monomials in plethysms $s_\mu[s_\lambda]$ and understanding how a (symmetric) sum of quasicrystals can be regrouped into connected components associated to Schur functions. To do the latter, we introduce an elegant algorithm. 
This algorithm is 
not more efficient than others curently in use, 
but its structure may help understand better plethysm. \\


Figure~\ref{fig:crystalsOfWordsAndTableaux} illustrates the corresponding connected components of crystals of $GL_n$ on words and tableaux, with words appearing as skew tableaux whose rows are their maximal weakly increasing factors. 


\begin{figure}[h]
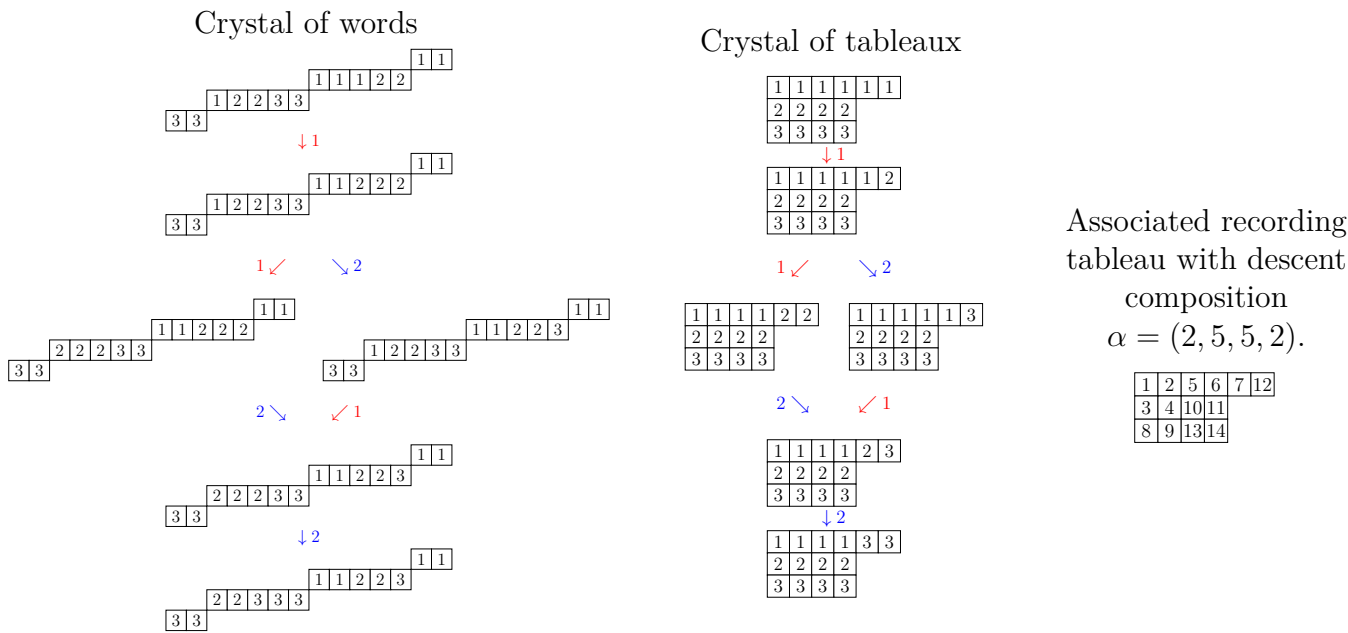

\Yboxdim{0.5cm}
\hspace{-6em}
\begin{minipage}{0.53\textwidth}
	\centering
	\vspace{-0.3em}
	Crystal of words
	
	\vspace{0.2em}
	
	\resizebox{\linewidth}{!}{
		\begin{tabular}{ccc}
			&\gyoung(::::::::::::;1;1,:::::::;1;1;1;2;2,::;1;2;2;3;3,;3;3)&\\
			& \textcolor{red}{$\downarrow 1$}&\\
			&\gyoung(::::::::::::;1;1,:::::::;1;1;2;2;2,::;1;2;2;3;3,;3;3)&\\
			&&\\
			&\textcolor{red}{$1 \swarrow$} $ \quad \quad $ \textcolor{blue}{$\searrow 2$} & \\
			&&\\ &
			\gyoung(::::::::::::;1;1,:::::::;1;1;2;2;2,::;2;2;2;3;3,;3;3) $\quad$ \gyoung(::::::::::::;1;1,:::::::;1;1;2;2;3,::;1;2;2;3;3,;3;3) &\\
			&&\\
			& \textcolor{blue}{$2 \searrow$} $ \quad \quad $ \textcolor{red}{$\swarrow 1$} &  \\
			&&\\
			&\gyoung(::::::::::::;1;1,:::::::;1;1;2;2;3,::;2;2;2;3;3,;3;3)&\\
			& \textcolor{blue}{$\downarrow 2$}&\\
			&\gyoung(::::::::::::;1;1,:::::::;1;1;2;2;3,::;2;2;3;3;3,;3;3)&\\
	\end{tabular}}
\end{minipage}\hspace{0.5em}%
\begin{minipage}{0.29\textwidth}
	\centering
	\vspace{-0.8em}
	Crystal of tableaux
	
	\vspace{0.5em}
	
	\resizebox{\linewidth}{!}{
		\begin{tabular}{ccc}
			&\young(111111,2222,3333)&\\
			& \textcolor{red}{$\downarrow 1$}&\\
			&\young(111112,2222,3333)&\\
			&&\\
			&\textcolor{red}{$1 \swarrow$} $ \quad \quad $ \textcolor{blue}{$\searrow 2$} & \\
			&&\\ &
			\young(111122,2222,3333) $\quad $ \young(111113,2222,3333) & \\
			&&\\
			& \textcolor{blue}{$2 \searrow$} $ \quad \quad $ \textcolor{red}{$\swarrow 1$} &  \\
			&&\\
			&\young(111123,2222,3333)&\\
			& \textcolor{blue}{$\downarrow 2$}&\\
			&\young(111133,2222,3333)&\\
	\end{tabular}}
\end{minipage}\hspace{0.5em}%
\begin{minipage}{0.29\textwidth}
	\centering
	Associated recording tableau
	with descent composition $\alpha = (2,5,5,2)$.
	
	\vspace{0.5em}
	\Yboxdim{0.5cm}
	\resizebox{0.4\linewidth}{!}{\young(12567<12>,34<10><11>,89<13><14>)}
\end{minipage}%
\Yboxdim{0.4cm}
\caption{Isomorphic connected components of crystals on words and tableaux \label{fig:crystalsOfWordsAndTableaux}}
\end{figure}

\section{Background}
\label{section:Background}

\subsection{Crystals on words}
Recall that a \emph{word (of length $k$ on $n$)} is a sequence of $k$ letters $w\in [n]^{\otimes k}$ on the alphabet $[n]=\{1,2,\ldots,n\}$. We say a word $w$ has \emph{weight} $\beta=(\beta_1,\beta_2,\ldots,\beta_n)$ where $\beta_i$ counts occurences of the letter $i$ in $w$.\\

One can define \emph{crystal operators $e_i$ and $f_i$ on words}, for $1\leq i < n$. A general description is that \emph{$e_i$} changes a letter $i+1$ into an $i$, or is null, as \emph{$f_i$} changes an $i$ into an $i+1$, or is null, and if $e_i$, $f_i$ are not null, then $e_i\circ f_i = f_i\circ e_i = Id$.
The exact action of $e_i$ and $f_i$ on words is described using the following \emph{parenthesis rule}: 
\begin{enumerate}
	\item Each letter $i$ is associated to a parenthesis ), and each letter $i+1$, to a parenthesis (.
	\item Coupled parenthesis are removed to obtain a sequence of uncoupled parenthesis 
	$)^{\phi_i}(^{\epsilon_i}$.
	\item $e_i$ acts on the letter $i+1$ corresponding to the leftmost uncoupled parenthesis (, and 
	$f_i$ acts on the letter $i$ corresponding to the rightmost uncoupled parenthesis ).
\end{enumerate}

\begin{ex}
	Let $w = 1331233312233$, and $i = 1$.\\
	
	The associated sequence of parenthesis is
	\resizebox{6cm}{!}{
	\begin{tabular}{ccccccccccccc}
	1&3&3&1&2&3&3&3&1&2&2&3&3\\
	)&&&)&(&&&&)&(&(&&\\	
	\end{tabular}}, which is reduced to \resizebox{6cm}{!}{
	\begin{tabular}{ccccccccccccc}
	1&3&3&1&2&3&3&3&1&2&2&3&3\\
	)&&&\textcolor{red}{)}&&&&&&\textcolor{red}{(}&(&&\\	
\end{tabular}} by removing coupled parenthesis\\. 
Then $f_1(w)=133\textcolor{red}{2}233312233$ and $e_1(w)=1331223331\textcolor{red}{1}233$.
\label{ex:parenthesisRule}
\end{ex}

A \emph{crystal on words} is then a labelled oriented graph on words where the edges $w\rightarrow w'$ are labelled by $i$ if $f_i(w) = w'$. See figure~\ref{fig:crystalsOfWordsAndTableaux} for an example of a 
crystal on words.

\subsection{Crystals on tableaux}

\subsubsection{Tableaux and Schur functions}
Recall that \emph{partitions} are weakly decreasing positive integer vectors $\lambda = (\lambda_1,\lambda_2, \ldots, \lambda_\ell)$. If the sum of the parts of $\lambda$ gives $m$, then we say that $\lambda$ is a \emph{partition of $m$}, noted $\lambda\vdash m$ or $|\lambda|=m$, and use $\ell(\lambda)=\ell$ to denote its number of parts, or length. We identify the partitions with their \emph{Young diagram}, the top- and left-justified array of boxes with $\lambda_i$ boxes in the $i^{th}$ row. A \emph{tableau of shape $\lambda$} is a filling of the cells of $\lambda$ with integers. We say that a tableau is \emph{semistandard} if the entries weakly increase along rows from left to right, and increase down columns. A tableau is said to be \emph{standard} if it is semistandard and entries $1$ to $|\lambda|$ appear exactly once. Unless otherwise stated, we use the word tableau for semistandard tableau.\\

The \emph{weight} (or \emph{filling}) of a tableau is the composition $\gamma = (\gamma_1,\gamma_2,\ldots)$ where $\gamma_i$ counts its entries $i$. To a tableau it is possible to associate a word by using a fixed reading order. 
We use the \emph{row reading order}, noted $rw$: we read rows from left to right, starting with the last row, and ending with the first.

\begin{ex}$t =$ \young(11234,2235,34,5)
	\hspace{-0.5em} is a tableau of shape $(5,4,2,1)$ and weight $(2,3,3,2,2)$.
	
	Its row reading word is $rw(t) = 534223511234$. 
\end{ex}

To a tableau $t$ of weight $\gamma$, it is also possible to associate the monomial $x^t=x_1^{\gamma_1}x_2^{\gamma_2}x_3^{\gamma_3}\ldots$. This gives the connection between tableaux and \emph{Schur functions}: the Schur function associated to a partition $\lambda$ is $\displaystyle s_\lambda=\sum_{t\in SSYT(\lambda)} x^t$, where $SSYT(\lambda)$ is the set of all tableaux of shape $\lambda$.\\

The cells containing entries $i$ in a tableau form a \emph{horizontal band}: each column contains at most one entry $i$, and reading the tableau left to right, each new cell with content $i$ must be weakly North-East (NE) to the preceading ones. We say the NE-most cell of a horizontal band is its head, and its SW-most cell, its tail. 
We say a horizontal band (containing a certain number of entries, up to entries $i$) is \emph{maximal} if adding the "next" horizontal band (of entries $i+1$) is not a horizontal band anymore.\\

We call the set of maximal horizontal bands of a tableau its \emph{minimal parsing}, and say it has \emph{type $\alpha$} if the length of the $i^\text{th}$ maximal horizontal band is given by $\alpha_i$. Among the tableaux with a fixed minimal parsing of type $\alpha$, there is a unique one which also has weight $\alpha$, obtained by filling the $i^{\text{th}}$ maximal horizontal band by entries $i$, for all $i$. We denote these tableaux $T_\alpha$. 

\begin{ex}
	Let $T=$ \gyoung(;1;1!\rd;3!\bl;5,!\rd;2;3!\bl;4,;4,!\wt;7) and $T'=$ \gyoung(;1;1!\rd;2;2!\bl;3;3,!\rd;2!\bl;3;3;3,;3), where the maximal horizontal bands of the two tableaux are distinguished.\\ 
	
	$T$ has minimal parsing of type $(2,3,3,1)$, and its horizontal bands (of individual entries $i$) are not all maximal. The tableau $T_{(2,3,3,1)}$ with same minimal parsing appears in example~\ref{ex:StdTab}. For its part, $T'$ has minimal parsing of type $(2,3,6)$, and all its horizontal bands are maximal. It is then equal to $T_{(2,3,6)}$ for this specific minimal parsing of type $(2,3,6)$. 
	\label{ex:maxHorBands}
\end{ex}

The integer vector $\alpha = (\alpha_1,\alpha_2,\ldots,\alpha_s)$ giving the length of the maximal horizontal bands in a tableau is what we'll call its \emph{descent composition}. This corresponds to the traditional notion of descent composition (in standard tableaux), as we see below.\\

Using the minimal parsing allows us to \emph{standardize a tableau}: entries in the first maximal horizontal band are relabelled by $1$ to $\alpha_1$, the ones in the second, by $\alpha_1+1$ to $\alpha_1+\alpha_2$, etc. This gives the same result as standardizing a tableau through its reading word, as seen below.\\

For standard tableaux, we can consider their \emph{descents}: entries $i$ such that $i+1$ appears in a row of greater index. To descent sets, we can associate bijectively a \emph{descent composition}: if $\{i_1<i_2<\ldots<i_k\}$ is the descent set of a standard tableau with $m$ cells, then $\alpha = (i_1,i_2-i_1,i_3-i_2,\ldots, m-i_k)$ is the associated descent composition. It is a composition of $m$, and gives the lengths of the maximal horizontal bands of the standard tableau. There is then a unique semistandard tableau with weight $\alpha$ and same minimal parsing as the standard tableau. 
This gives the following.

\begin{prop}
	Semistandard tableaux with minimal parsing of type and weight $\alpha$ are in bijection with standard tableaux with descent composition $\alpha$.
	\label{prop:bijTabStdTabBandesMax}
\end{prop}

\begin{ex}
	The tableau $T_{(2,3,3,1)}=\young(1123,223,3,4)$ standardizes to $std(T_{(2,3,3,1)})=\young(1258,347,6,9)$. This standard tableau has descent set $\{2,5,8 \}$, and descent composition $(2,3,3,1)$. The descent composition then gives the lengths of the maximal horizontal bands in $T_{(2,3,3,1)}$ and $std(T_{(2,3,3,1)})$. Note that all tableaux with the same minimal parsing of type $(2,3,3,1)$, like the tableau $T$ in example~\ref{ex:maxHorBands}, standardize to $std(T_{(2,3,3,1)})$. 
	
	\label{ex:StdTab}
\end{ex}

We will see in section~\ref{section:QSym} that descent compositions are used to define fundamental quasisymmetric functions, which are central to our study.\\

The \emph{descents of a word $w$} are the positions $i$ such that $w_i>w_{i+1}$. The descent composition 
of a word $w$ corresponds to the lengths of maximal weakly increasing factors in $w$. Words can then be standardized uniquely in a way that preserves descents: entries $i$ of $w$ are replaced from left to right by entries $\beta_1+\beta_2+\ldots+\beta_{i-1}+1$ to $\beta_1+\beta_2+\ldots+\beta_{i}$, where $\beta_j$ counts letters $j$ in $w$. 
A tableau $T$ can be standardized by standardizing its reading word $rw(T)$, which gives the same result as above. 
%
%
%

\subsubsection{Crystals of tableaux}

\emph{Crystal operators } on tableaux are defined as applying the (word) crystal operators on its reading word, and changing the corresponding entry in the tableau. Tableaux obtained in this way are always semistandard \cite{BumpSchilling}. These crystal operators define an oriented graph structure on the set of semistandard tableaux, where there is an arrow from $T$ to $T'$ labelled $i$ if $T'=f_i(T)$.\\

Since only entries change and the shape is fixed, the connected components regroup all the tableaux of the same shape $\lambda$ which we denote $B(\lambda)_n$ if the fixed maximal entry is $n$. It then corresponds to the irreducible character $\chi^\lambda$ of $GL_n$ given by the Schur function $s_\lambda(x_1,x_2,\ldots,x_n)$. More generally, we can also consider the infinite graph $B(\lambda)$ corresponding to $s_\lambda$. 
The unique source vertex of $B(\lambda)$ (and any $B(\lambda)_n$) is the tableau of shape and filling $\lambda$, which we denote $1_\lambda$. Note that it corresponds to $T_{\lambda}$.\\

For example, $1_{(5,4,2)} = T_{(5,4,2)} =$ \young(11111,2222,33).\\

Crystals of tableaux are especially important to study, because any connected crystal of type $A_{n-1}$ is isomorphic to a crystal of tableau:
\begin{theo}[\cite{BumpSchilling}] For any connected Stembridge crystal $C$ for $GL_n(\C)$ (of type $A_{n-1}$), there is a unique source. Its weight is a partition $\lambda$ and $C \isom B(\lambda)$.
\label{prop:IsomCristalTab}
\end{theo}

Crystals of words and of tableaux are strongly linked through the Robinson-Schensted-Knuth (RSK) algorithm, jeu de taquin, and the plactic and coplactic monoids (see section~\ref{section:RSKPlacticCoplactic}).

\subsection{Fundamental quasisymmetric functions and descent compositions}
\label{section:QSym}

The ring of quasisymmetric functions $QSym$, introduced by Gessel, generalizes and contains the ring of symmetric functions \cite{Gessel}. We will consider the basis of $QSym$ given by the \emph{fundamental quasisymmetric functions}, which are indexed by compositions: 
\[
F_\alpha = \sum_{\alpha\preccurlyeq \beta} M_\beta \text{, where } M_\beta = \sum_{i_1<i_2<\ldots<i_k} x_{i_1}^{\beta_1} x_{i_2}^{\beta_2} \ldots x_{i_k}^{\beta_k},
\]
and $\alpha\preccurlyeq \beta$ indicates that $\beta$ is a refinement of $\alpha$: adjacent parts of $\beta$ can be summed to obtain $\alpha$. The $M_\beta$ are \emph{monomial quasisymmetric functions}, and also form a basis of $QSym$. \\

For example, $\beta_1=(2,1,3,2,4,1)$ and $\beta_2=(1,1,3,1,2,1,1,1,1,1,)$ are distinct, but incomparable, refinements of $\alpha=(2,4,2,5)$. Therefore $M_{\beta_1}$ and $M_{\beta_2}$ both appear in $F_\alpha$.\\

Schur functions (which are also quasisymetric functions) then decompose in the basis of
fundamental quasisymmetric functions. We will use the decomposition below, proved recently by Gessel in a short article~\cite{Gessel2019}, by using horizontal bands in standard tableaux, there called runs, and an involution acting on them. 

\begin{equation*} s_\lambda = \sum_{T\in SYT(\lambda)} F_{DesComp(T)}.
\tag{$\ast$}
\label{eq:Gessel}
\end{equation*}

\section{Decomposing a crystal into subcomponents corresponding to fundamental quasisymmetric function}

We now show how the above formula~(\ref{eq:Gessel}) expressing $s_\lambda$ as a sum of fundamental quasisymmetric functions $F_\alpha$ induces a decomposition of the crystal of tableaux $B(\lambda)$. Each subcomponent of the decomposition would then correspond precisely to one of the $F_\alpha$ appearing in~(\ref{eq:Gessel}). \\

We remark that crystal operators on tableaux do not necessarily preserve descent compositions. This is because changing the value of entries can modify the maximal horizontal bands along with the weight of the tableau. However, tableaux of the same parsing will be grouped together in connected subcomponents of $B(\lambda)$: 

\begin{prop}
	The set of tableaux of shape $\lambda$ with a fixed parsing of type $\alpha$ form a connected induced subgraph of $B(\lambda)$. 
	Its source is the tableau $T_\alpha$ with filling and (same) minimal parsing of type $\alpha$, and its vertices give the monomials appearing in $F_\alpha$.
	\label{thm:SourcesQuasiSym}
\end{prop}

An example of the decomposition can be seen in figure~\ref{figure:CrystalDecompInQuasiSymFcts}.

\begin{proof}
	Let's consider the definition of $F_\alpha$. If $T_\alpha$ is a tableau of filling and parsing into horizontal bands of type $\alpha$, then $x^{T_\alpha}=x^\alpha$ appears in $M_\alpha\subseteq F_\alpha$, since $\alpha\preccurlyeq \alpha$. Now, any refinement $\alpha\preccurlyeq \beta$ gives a (non-)minimal parsing of the same maximal horizontal bands. Any filling $\gamma$ obtained from $\beta$ by (potentially) adding zero parts gives a valid filling of the same (non-)maximal horizontal bands, and the associated monomial will appear in $M_\beta$. In particular, $\alpha\preccurlyeq \beta\preccurlyeq \gamma $. Therefore, the set of tableaux of shape $\lambda$ with fixed minimal parsing of type $\alpha$ (and any weight $\alpha\preccurlyeq \gamma$) gives all monomials of $F_\alpha$.\\
	
	Crystal operators $f_i$, for $1\leq i\leq n-1$, modify the weight of tableaux by $b_{i+1}-b_{i}$, where $b_i$ is the vector with zeros everywhere except in position $i$. If a tableau $T$ of weight $\gamma$ has minimal parsing of type $\alpha$, then $f_i(T)$ has the same parsing if and only if $\alpha\preccurlyeq \gamma+(b_{i+1}-b_{i})$. This follows from the above discussion.\\
	
	The set of tableaux with the same minimal parsing, and so the same descent composition $\alpha$, form a subset of the vertices of $B(\lambda)$. Among these tableaux, there is only one with filling $\alpha$, $T_\alpha$. We will now show that every tableau in the subset of vertices can be obtained by a certain sequence of crystal operators from $T_\alpha$, thus showing that the induced subgraph is connected, and that $T_\alpha$ is its source. \\
	
	Let $T$ be a tableau of weight $\gamma$ with the same parsing of type $\alpha$, then there is a set of sets of consecutive parts of $\gamma$ such that the sum of the parts in every set gives a part of $\alpha$.
	Let \[\{ \{ \gamma_1,\gamma_2,\ldots,\gamma_{k_1}\}, \{\gamma_{k_1+1},\ldots,\gamma_{k_2} \}, \ldots, \{ \gamma_{k_{s-1}+1}, \ldots, \gamma_{k_{s}} \} \}\] be such a set of sets, with $k_1<k_2<\ldots<k_s=\ell(\gamma)$, so $\displaystyle \sum_{j=k_{r-1}+1}^{k_r} \gamma_j = \alpha_r$ for $1\leq r\leq s=\ell(\alpha)$.\\
	
	Then the following sequence of crystal operators applied to $T_\alpha$ gives $T$:
	\begin{center}
	$(f_1)^{\gamma_2}\circ (f_2\circ f_1)^{\gamma_3}\circ (f_3\circ f_2\circ f_1)^{\gamma_4}\circ \ldots\circ (f_{k_1-1}\circ \ldots\circ f_2\circ f_1)^{\gamma_{k_1}} \circ \ldots \circ$\\
	$(f_{k_{(j-1)}}\circ \ldots\circ f_{j+1}\circ f_j)^{\gamma_{k_{(j-1)}+1}} \circ \ldots \circ (f_{k_{j}-1}\circ \ldots\circ f_{j+1}\circ f_j)^{\gamma_{k_{j}}} \circ \ldots \circ$\\
	$(f_{k_{(s-1)}}\circ \ldots\circ f_{s+1}\circ f_s)^{\gamma_{k_{(s-1)}+1}} \circ \ldots \circ (f_{k_{s}-1}\circ \ldots\circ f_{s+1}\circ f_s)^{\gamma_{k_{s}}} \quad (T_\alpha) \quad =\quad  T.$
	\end{center}

This sequence changes entries in the last horizontal band first, then in the previous to last, etc. until the entries in the first maximal horizontal band are changed, and the obtained tableau is $T$.
Moreover, it is straightforward to see that every intermediate tableau also has the same minimal parsing into horizontal bands.\\
	
	Finally, the labelled oriented subgraph on tableaux with the same minimal parsing of type $\alpha$, with labels of edges given by the application of crystal operators which preserve minimal parsing, gives a connected \textit{induced} subgraph of $B(\lambda)$. This is because crystal operators which preserve the minimal parsing remain crystal operators, and if there is an edge between two tableaux in the subset, then the crystal operator applied preserves the minimal parsing.
	%
\end{proof}

\begin{rem}
	The induced subgraphs are not generally crystals. They are however isomorphic (as oriented graphs) to crystals $B(\mu)$ after re-labelling of vertices and oriented edges, as we will see in section~\ref{section:Subcrystals}. 
\end{rem}

We denote the induced subgraphs of $B(\lambda)$ with minimal parsing of type $\alpha$ 
by $B(T_\alpha)$, where the tableaux $T_\alpha$ are the source vertices. Note that there may be many subcomponents associated to the same composition $\alpha$, which means that $F_\alpha$ occurs more than once in $s_\lambda$. In particular, this is simply because there can be many ways to decompose $\lambda$ into maximal horizontal bands of respective lengths $\alpha_i$. We will see in the next section that all subcomponents associated to the same composition $\alpha$ are in fact isomorphic.

\begin{theo*}
	The connected crystal $B(\lambda)_n$, of tableaux of shape $\lambda$ and maximal entry $n$, is partitionned into disjoint connected induced subgraphs $B(T_\alpha)$ which correspond to quasisymmetric functions $F_\alpha(x_1,\ldots,x_n)$, where the subsets of vertices are tableaux with a fixed descent composition $\alpha$. The sources $T_\alpha$ of these subcomponents have filling and minimal parsing of type $\alpha$. 
	The number $f^\lambda_\alpha$ of subgraphs of type $\alpha$ is the number of standard tableaux of shape $\lambda$ and descent composition $\alpha$.
	\label{theo:decompBlambdaIntoSubcomponents}
\end{theo*}

\begin{ex}
	Figure~\ref{figure:CrystalDecompInQuasiSymFcts} shows the decomposition of $B(4,3)_4$, into subcomponents associated to quasisymmetric functions. Those associated to a descent composition $\alpha$ appear in the same color as that associated to $\overset{\leftarrow}{\alpha}$. We show in section~\ref{section:dualposition} that pairs of subcomponents associated respectively to descent compositions $\alpha$ and $\overset{\leftarrow}{\alpha}$ are positionned dualy in the crystal. 

	\begin{figure}
		\hspace{-5.5em}
		\includegraphics[width=1.25\linewidth]{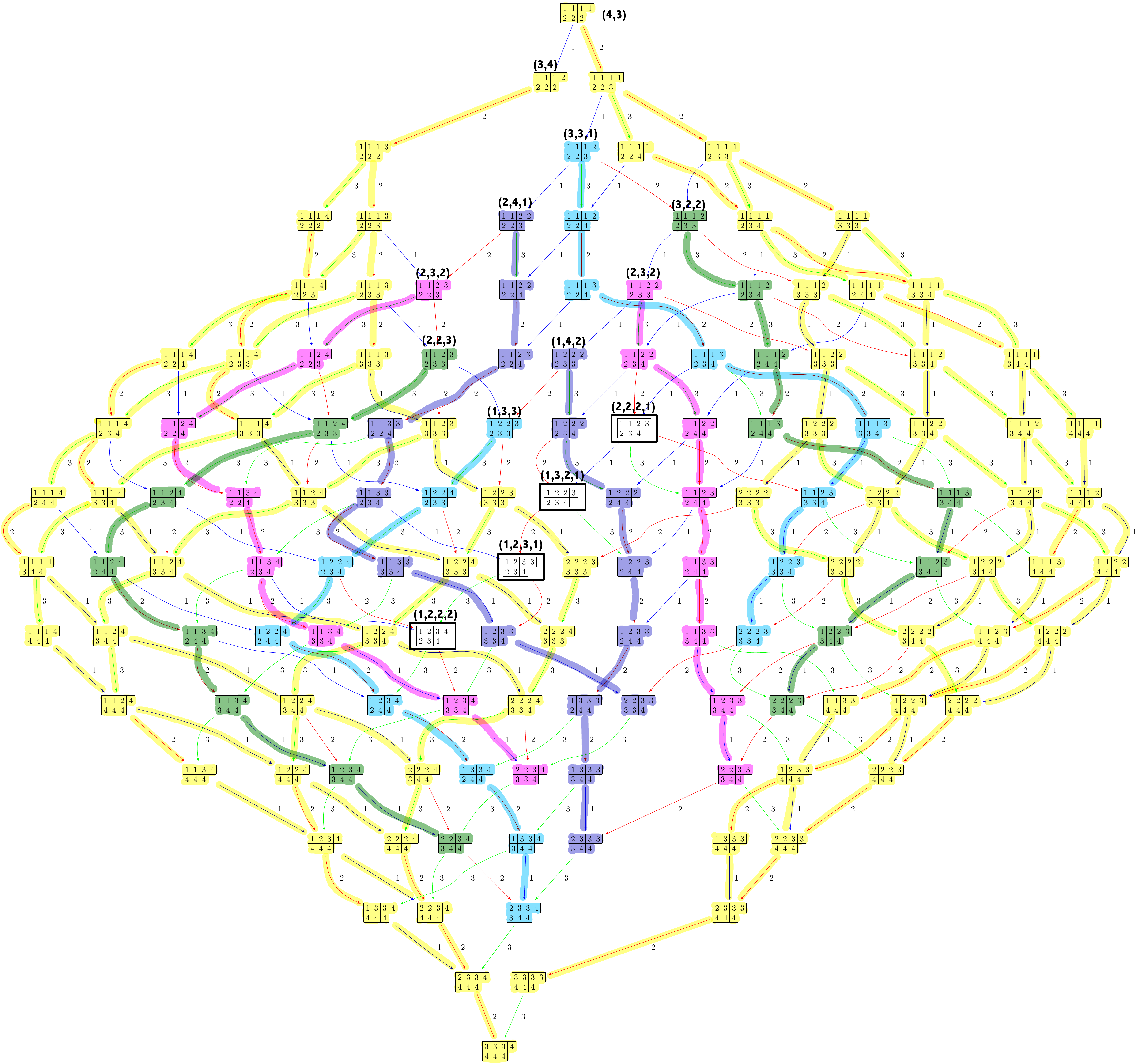}
		\caption{Decomposition of $B(4,3)_4$ into its subcomponents associated to fundamental quasisymmetric functions $F_\alpha$:  $s_{(4,3)} = F_{(4,3)}+F_{(3,4)}+F_{(3,3,1)}+F_{(2,4,1)}+F_{(3,2,2)}+2\cdot F_{(2,3,2)}$\newline $+F_{(2,3,3)}+F_{(2,2,3)}+F_{(1,4,2)}+F_{(1,3,3)}+F_{(2,2,2,1)}+ F_{(1,3,2,1)}+F_{(1,2,3,1)}+F_{(1,2,2,2)}$. \label{figure:CrystalDecompInQuasiSymFcts}}
	\end{figure}
\end{ex}


\begin{proof}[Proof of theorem~\ref{theo:decompBlambdaIntoSubcomponents}]
	The minimal parsing is uniquely determined for any tableau, so the sets of tableaux with given minimal parsing are disjoint. As we have seen above, each of these sets induce a connected subgraph of $B(\lambda)$ with source $T_\alpha$, and the monomials associated to these tableaux form a quasisymmetric functions $F_\alpha$.
	Finally, the sources $T_\alpha$ are put in bijection with standard tableaux of descent compositions $\alpha$ and same minimal parsing by using proposition~\ref{prop:bijTabStdTabBandesMax}.
	By the formula of Gessel,
	this confirms that we get the right number of subcomponents associated to each $F_\alpha$ in $B(\lambda)$.
\end{proof}

\begin{prop}
	Compositions $\alpha=(\alpha_1,\alpha_2,\ldots,\alpha_s)$ which appear as descent compositions of tableaux in $B(\lambda)_n$, for $\lambda= (\lambda_1, \lambda_2,\ldots \lambda_\ell)\vdash m$, have the following properties.

	\begin{enumerate}[itemsep=0ex]
		\item $1\leq \alpha_i\leq\lambda_1$,
		\item $\alpha_1+\alpha_2+\ldots + \alpha_j\leq \lambda_1+ \lambda_2 + \ldots + \lambda_j$ for all $j$,
		\item $s\leq (\lambda_2+\lambda_3+\ldots + \lambda_\ell) + 1$,
		\item $\ell\leq s\leq n$,
		\item $s\leq k$.
		\end{enumerate}
	\label{prop:conditionsOfDescentCompositionsAndPartitions}
\end{prop}

\begin{proof}
\begin{enumerate}[itemsep=0ex]
\item The $\alpha_i$'s describe lengths of maximal horizontal bands, and each has at most one cell in every column of $\lambda$. Then $\alpha_i\leq \lambda_1$, since $\lambda_1$ is the number of columns spanned by $\lambda$. Moreover, by the maximality of (maximal) horizontal bands, $\alpha_i\geq 1$.

\item In order to be maximal, the horizontal band of $i$'s must have its tail on a row of larger index than that on which lies the head of the horizontal band of $(i+1)$'s. Then the $j$ first horizontal bands span at most $j$ rows and their cells.

\item The first horizontal band is necessarily of shape $(\alpha_1)$. This gives us the $+1$. The maximal number of horizontal bands possible occurs if every new horizontal band has a single cell on the next non-fully filled row, and the rest of its cells on the rows above.

\item The first inequality follows from the discussion on the second condition. The second inequality follows from the fact that the composition $\alpha$ also gives the weight of a tableau $T_\alpha$ of shape $\lambda$ by the above proposition. Then $s\leq n$, since $n$ is the maximal entry allowed in tableaux of $B(\lambda)_n$, as a crystal of $GL_n$, and $\alpha_s$ indicates the number of entries $s$ in a source $T_\alpha$.

\item There cannot be more non-zero parts to the weight of a tableau than the number of cells in it.\qedhere
\end{enumerate}


%
\end{proof}

\begin{rem}
	For $\lambda\vdash m$ fixed, not all the compositions of $m$ with the properties above are descent composition for $\lambda$.
	For example, for $\lambda=(3,3)$, $\alpha=(1,2,3)$ has all the properties above, however, there are no semistandard tableau of shape $\lambda$ with minimal parsing $\alpha$.  
	
\end{rem}

\begin{conj}
	Let $\lambda$ be the weakly decreasing reordering of the parts of $\alpha$. Then $F_\alpha$ occurs in $s_\lambda$. In particular, all reoderings $\alpha$ of $\lambda$ appear as descent compositions for $\lambda$, with $F_\alpha$ in $s_\lambda$.
\end{conj}

This has been tested, and found to be true, for all compositions of $m$, untill $m=6$.

\section{Quasicrystal structure of a subcomponent associated to a fundamental quasisymmetric function $F_\alpha$}
\label{section:Subcrystals}

We have seen in theorem~\ref{theo:decompBlambdaIntoSubcomponents} that subcomponents $B(T_\alpha)$ associated to a fundamental quasisymmetric function $F_\alpha$ form induced subgraphs of crystals. This notion of induced subgraph of crystal has not been studied by the mathematical community to our knowledge.\\

We push this further by studying the structure of these subcomponents. In this section, we prove that all subcomponents associated to a composition $\alpha$ are isomorphic, no matter their crystal host $B(\lambda)_n$. We denote that class of subcomponents by $B(\alpha)_n$, and call them \emph{quasicrystals}.\\
 
The quasicrystals $B(\alpha)_n$ are not crystals of type $A_{n-1}$, one reason being that they are not self-dual in general. They may however be Kashiwara crystals for other groups. It would be interesting to investigate which groups (and the associated representations) might have such a crystal structure. 


\subsection{Oriented graph structure of quasicrystals $B(\alpha)_n$}

\begin{theo*}
	Let $\alpha$ be a composition of $m$ in $s$ parts. The quasicrystal $B(\alpha)_n$ is isomorphic (as an oriented graph) to $B(m)$ with maximal entry $n-s+1$. In particular, the oriented graph structure of $B(\alpha)_n$ is independant of partitions $\lambda$ for which $\alpha$ is a descent composition.
	\label{theo:IsomWithBm}
\end{theo*}

In order to prove this theorem, we need to introduce the original definition of quasisymmetric functions of Gessel \cite{Gessel}, which is in terms of subsets $I\subseteq\{1,2,\ldots,m-1\}$. To do this, we use the bijection between descent sets $I$ and descent compositions $\alpha$ defined before proposition~\ref{prop:bijTabStdTabBandesMax}: if $\alpha = (\alpha_1,\ldots,\alpha_s)$, then the associated set is $I_\alpha=\{j_1,j_2,\ldots,j_{s-1}\}$ where $j_i=\alpha_1+\ldots+\alpha_i$. Then
%
\[F_\alpha(x_1,\ldots,x_n) = F_{I_\alpha} (x_1,\ldots,x_n) = \sum_{\substack{1\leq i_i\leq i_2\leq \ldots\leq i_m\leq n \\ \text{ with } i_j<i_{j+1} \text{ if } j\in I_\alpha}} x_{i_1} x_{i_2} \ldots  x_{i_m}.\]

This definition is more generally used in the literature as the one given in the introduction. 

\begin{proof}[Proof of theorem~\ref{theo:IsomWithBm}]
	A subcomponent $B(T_\alpha)$ in any crystal $B(\lambda)_n$ corresponds to the quasisymmetric function $F_\alpha(x_1,\ldots,x_n)$, since we restrict ourselves to a maximal entry $n$. The tableaux in the crystal $B(m)_{n-s+1}$ have a unique descent composition, $\alpha=(m)$. Then, the whole crystal corresponds to the quasisymmetric function $F_{(m)}(x_1,\ldots,x_{n-s+1})$. Lets start by showing that the sets of monomials associated to both quasisymmetric functions above are in bijection. To do this, we use the original definition of quasisymmetric functions.\\
	
	Let's note that, for $1\leq s\leq n$, \[F_{(m)}(x_1,\ldots,x_{n-s+1}) = F_{\varnothing} (x_1,\ldots,x_{n-s+1}) = \sum_{1\leq i_i\leq i_2\leq \ldots\leq i_m\leq n-s+1 } x_{i_1} x_{i_2} \ldots  x_{i_m}.\]
	The weakly increasing sequence of integers which index each monomial of $F_{\varnothing}(x_1,\ldots,x_{n-s+1})$ is in bijection with the weakly increasing sequence of integers indexing the monomials of any $F_{I}$, with $I=\{j_1,\ldots,j_{s-1}\}$, through the following bijection. 
	
	\begin{center}
		$1\leq i_i\leq i_2\leq \ldots\leq i_m\leq n-s+1 $ \\
		$\downarrow$\\
		$1\leq i_i\leq \ldots \leq i_{j_1} < i_{j_1+1}+1\leq \ldots \leq i_{j_2}+1 < i_{j_2+1}+2 \leq \ldots \leq  i_{j_k}+(k-1) < i_{j_k+1}+k \leq \ldots \leq i_{j_{s-1}}+(s-2) < i_{j_{s-1}+1} + (s-1) \leq \ldots \leq i_m +(s-1) \leq (n-s+1)+(s-1) =n. $
	\end{center}
	
	All possible sequences of indices in $f_I$ can be retreived this way, and the strict ascents will be respected. One can find the initial sequence by removing $i$ from each part between the $i^{\text{th}}$ and the $i+1^{\text{th}}$ increasing sign $<$. Similarly as before, all sequences of $F_\varnothing$ can be retreived this way.\\
	
	
	Lets now fix $n,m$, $1\leq s\leq n$, a composition $\alpha$ of $m$ in $s$ parts, and any shape $\lambda\vdash m$ in which $\alpha$ appears as a descent composition. 
	The sequences of indices above then give fillings of the maximal horizontal bands of type $\alpha$ in $\lambda$: indices $i_j$ indexed by $1 \leq j \leq j_1$ fill the first horizontal band, indices $i_j$ indexed by $j_1+1 \leq j \leq j_2$ fill the second horizontal band, etc.\\
	
	Let us replace the tableaux in the crystal $B(m)_{n-s+1}$ by the corresponding tableaux 
	according to the above bijection. Then the first tableau has weight and minimal parsing of type $\alpha$, i.e. if indexing the integers by their position in the sequence, we get \[1\leq 1_1\leq \ldots \leq 1_{j_1} < 2_{j_1+1} \leq \ldots \leq 2_{j_2}<3_{j_2+1}\leq \ldots \leq k_{j_{k}}<(k+1)_{j_{k}+1} \leq \ldots\leq s_{m}\leq n.\]
	
	Since the indices $i_j$ correspond to entries in a tableau, we can consider how crystal operators act on such entries. We are restricting ourselves to crystal operators which preserve the parsing, so they may only be applied to the indices $i_j$ if the weak order described above is preserved. In particular, the crystal operators which may be applied on the sequence, without breaking its weak order, are $f_{i_{j_k+\ell}+k}$ 
	if $f_{i_{j_k+\ell}}$ can be applied to the corresponding tableau 
	in $B(m)_{n-s+1}$.
	Then the structure of the quasicrystal $B(\alpha)$ will be exactly that of $B(m)_{n-s+1}$, 
	with some relabelled oriented edges.\\
	
	Finally, this is independent of $\lambda$, since only the weakly increasing sequence is important in the above isomorphism.
\end{proof}

\begin{cor}
	Let $n,m\in\N$ be fixed, and consider any composition $\alpha$ of $m$ in $s\leq n$ parts. Then the number of monic monomials in $F_\alpha(x_1,\ldots, x_n)$ is equal to the number of monic monomials in $F_{(m)}(x_1,\ldots,x_{n-s+1})$.
\end{cor}


\begin{cor}
	For $n$ fixed, and a fixed composition $\alpha$, all subcomponents $B(T_\alpha)_n$ are isomorphic as labelled oriented graphs, no matter the crystal $B(\lambda)_n$ they live in.
	\label{prop:GraphStructureSubcomponents}
\end{cor}

\begin{proof}
	We have seen that $B(\alpha)$ is isomorphic to $B(m)_{n-s+1}$, where $n$ is the fixed maximal entry in the tableaux of $B(\alpha)$, $m=|\alpha|$ and $s=\ell(\alpha)$. Moreover, the isomorphism seen in the proof above does not depend on the shape $\lambda$ of tableaux, and the modifications of the labels only depend on $\alpha$. Therefore, a crystal operator can be applied on all tableaux in a given position in different $B(\alpha)$'s, no matter their shapes.
\end{proof}

In other words, it is justified to study graphs (or quasicrystals) associated to quasisymmetric functions $F_\alpha$, as their oriented graph structure is determined for any fixed $n$. We could also use a notation $B(m,n,s)$ as only these values are important in defining the oriented graph structure: relabellings of crystals $B(m)$ with maximal entry $n-s+1$. In particular, all $B(\alpha)_n$ for any composition of $m$ in $s\leq n$ parts will have the same oriented graph structure: that of $B(m)_{n-s+1}$. \\



\begin{rem}
	We can consider how the crystal operators will act on such weakly increasing sequences. For one thing, a crystal operator $f_i$ will act on the rightmost $i$ in the sequence, as long as it does not break the increasing sequences: the entries $i$ can only appear in one horizontal band at the time, in order to preserve the strict increasingness of the sequence, and the rightmost will correspond to the $NE$ most entry $i$ in the tableau. 
	This agrees with the parenthesis rule. 
\end{rem}


\subsection{Height of quasicrystals, sources and sinks}
\begin{cor}
	The quasicrystals $B(\alpha)_n$ have height, or length of their maximal subchain, $m\cdot (n-s) +1$, where $n$ is the fixed maximal entry, $s=\ell(\alpha)$ and $|\alpha| = m$.
\end{cor}

\begin{proof}
	The subcomponents $B(T_\alpha)$ in any crystal $B(\lambda)$ contain the chain of tableaux with transformations given by $(f_{n-s}^{\alpha_1} \circ \ldots \circ f_{2}^{\alpha_1} \circ f_{1}^{\alpha_1})
	\circ \ldots \circ 
	(f_{n-2}^{\alpha_{s-1}} \circ \ldots \circ f_{s}^{\alpha_{s-1}} \circ f_{s-1}^{\alpha_{s-1}}) \circ 
	(f_{n-1}^{\alpha_s} \circ \ldots \circ f_{s+1}^{\alpha_s} \circ f_s^{\alpha_s})$, which modifies maximally one horizontal band at a time, from its head to its tail: all entries $s$ are changed to $s+1$'s, then into $s+2$'s, etc. until they are all changed into $n$'s. Then all entries $s-1$ are changed into $s$'s, then into $S+1$'s, etc. until they have all been changed into $n-1$'s. This process is continued until all entries $1$ 
	have been changed into $n-s$'s and no more transformations can be applied without modifying the minimal parsing into horizontal bands.\\
	
	This sequence of transformations preserves the parsing into horizontal bands, so we remain always in the same subcomponent $B(T_\alpha)$. Moreover, no crystal operator can be applied to the final tableau of the chain without coming out of the subcomponent. We have then obtained, and described, the sink of the subcomponent $B(T_\alpha)$: preserving the same minimal parsing as the source, entries $i$ are replaced by $n-s+i$. This tableau has weight $(0^{n-s},\alpha)$.\\
	
	There are $|\alpha|\cdot (n-s)$ crystal operators in this sequence of transformations, and it describes a chain of length $|\alpha|\cdot (n-s) + 1$ in any subcomponent $B(T_\alpha)_n$, when adding the source to which the crystal operators are applied. Since the chain starts from the source and ends at the sink, it is maximal, since $f_i$ has a unique image going down each row of $B(\lambda)$.
	Then, any maximal chain in a quasicrystal $B(\alpha)$ will have the same length: $|\alpha|\cdot (n-s) + 1$.
	Since the oriented graph structure of $B(\alpha)$ is only determined by $n,m=|\alpha|$ and $s$, then the height $|\alpha|\cdot(n-s)+1=m\cdot(n-s)+1$, also only depends on $n,m,s$.
\end{proof}

\begin{cor}
	The sink of a subcomponent $B(T_\alpha)$ is obtained from its source $T_\alpha$ by replacing entries $i$ by $n-s+i$, where $s$ is the length of $\alpha$.
\end{cor}

\subsection{Number of semistandard tableaux of shape $\lambda$ and maximal entry $n$, and Kostka numbers}

Using the results above, we give a formula for computing the number of tableaux of a given shape $\lambda$ and maximal possible entry $n$. In other words, we count the number of tableaux in $B(\lambda)_n$, which is equal to the number of monic monomials in $s_\lambda(x_1,x_2,\ldots,x_n)$.\\

We have seen that $s_\lambda = \sum_{T\in SYT(\lambda)} F_{DesComp(T)}$, and that each quasisymmetric function $F(\alpha)$ corresponds to a subcomponent $B(T_\alpha)$ in the crystal $B(\lambda)_n$, whose vertices are all tableaux of shape $\lambda$ and maximal entry $n$.\\

The number of standard tableaux of shape $\lambda$ is well known to be given by the hook-length formula, and are generally not too hard to enumerate. We have seen that all $B(\alpha)_n$ for compositions $\alpha$ of $m$ with the same number $s\leq n$ of parts are isomorphic to $B(m)_{n-s+1}$, and that these subcomponents of $B(\lambda)_n$ are counted by the number of standard tableaux with $d=s-1$ descents. It would then suffice to have formulas for the number of tableaux in $B(m)_{n-d}$ and for the number of standard tableaux of shape $\lambda$ with $d$ descent to give a formula for the number of semistandard tableaux of shape $\lambda$.

Using the results above allows us to do this.

\begin{prop}
	The number of tableaux in $B(m)_{k}$ is equal to the multiset coefficient and binomial coefficient below.
	\[
	\left(\left(\begin{array}{c} m+1 \\ k-1 \end{array}\right)\right) = \left(\begin{array}{c} m+k-1 \\ k-1 \end{array}\right) = \frac{(m+1) \cdot (m+2)\cdot \ldots \cdot (m+k-1)}{(k-1)!}.
	\]
	\label{prop:NumberOfTableauxInBm}
\end{prop}

\begin{proof}
	For all tableaux of shape $(m)$, one must chose the position after which the array holds no more $1$'s, then the position after which the array holds no more $1$'s and $2$'s, etc. up to the end of the array, which is filled with $k$'s after the $k-1^{\text{th}}$ position. We want to allow repetitions of chosen positions, since we want to allow that an entry does not appear and is "sandwiched out". We also want to allow the tableau with only $1$'s, which is the source of $B(m)_k$, so we add $1$ position which lies outside of the array. Then if this outside position is picked $j$ times, for $1\leq j\leq k-1$, the entries $k-j+1$ to $k$ will not appear in the array. \\
	
	Therefore, one chooses $k-1$ positions in the $m+1$ possible ones for the breaks between integers, allowing for repetitions, and without keeping track of the order in which the positions are picked. This gives precisely the multiset coefficient above.
\end{proof}

\begin{ex}
	The arrays below have their corresponding multiset of positions of breaks noted under them, for $m=10$ and maximal entry $k=5$.
	\[
	\begin{array}{ccc} 
	\young(1111111333) & \young(1222224444) & \young(5555555555) \\
	\{ 8,8,11,11\} & 
	\{ 2,7,7,11\} &
	\{ 1,1,1,1\}
	\end{array}
	\]
\end{ex}

\begin{cor}
	The number of tableaux with maximal entry $n$ in a quasicrystal $B(\alpha)_n$, for any composition $\alpha$ of $m$ in $s$ parts, is given by
	\[
	\left(\left(\begin{array}{c} m+1 \\ n-s \end{array}\right)\right) = \left(\begin{array}{c} m+n-s \\ n-s \end{array}\right) = \frac{(m+1) \cdot (m+2)\cdot \ldots \cdot (m+n-s)}{(n-s)!}.
	\]
\end{cor}

\begin{cor}
	For any composition $\alpha\vDash m$ in $s$ parts, the number of monic monomials in $F_\alpha(x_1,\ldots, x_n)$ is equal to the binomial coefficient above.
\end{cor}

\begin{theo*}
	The number of tableaux of shape $\lambda$ with maximal entry $n$ is given by 
	\[|SSYT(\lambda)_n|=
	\sum_{0\leq d\leq D}
	f^\lambda_d\cdot
	\left(\begin{array}{c} |\lambda|+n-d-1 \\ n-d-1 \end{array}\right),
	\]
	where $f^\lambda_d$ denotes the number of standard tableaux of shape $\lambda$ with $d$ descents, and $D$ is the maximal number of descents in a standard tableau of shape $\lambda$.
	\label{theo:nbrTab}
\end{theo*}

\begin{ex}
	There are 14 standard tableaux of shape $(4,3)$, which all appear in figure~\ref{fig:graphTabStd}. Among these, two have one descents, 
	eight have two 
	and four have three. 
	
	Then the number of tableaux of shape $(4,3)$ with maximal entry $n$, for any $n$, is equal to \\
	$|SSYT((4,3))_n| = 2\cdot \left(\begin{array}{c} 7+n-1-1 \\ n-1-1 \end{array}\right) + 8\cdot \left(\begin{array}{c} 7+n-2-1 \\ n-2-1 \end{array}\right) + 4\cdot \left(\begin{array}{c} 7+n-3-1 \\ n-3-1 \end{array}\right)$.\\
	
	One may verify that when $n=4$, one retreives $140$ tableaux, the number of tableaux in figure~\ref{figure:CrystalDecompInQuasiSymFcts}. The following values for $n=5,6,7$ (for the same shape $(4,3)$) are respectively $560,1764$ and $4704$ tableaux, which shows how fast these numbers grow. This formula can then really help to enumerate tableaux of a given shape, especially when $n$ is large.
\end{ex}

\begin{proof}[Proof of theorem~\ref{theo:nbrTab}]
	Recall that a tableau with $d$ descents will have an associated descent composition in $s=d+1$ parts. Let then $D$ be the maximal number of descents in standard tableaux of shape $\lambda$, and let $S$ be the maximal number of parts of the associated descent compositions. \\
	
	When $n=D+1=S$, all connected components are present in $B(\lambda)_n$. All those associated to a descent composition $\alpha$ in $s=d+1$ parts is isomorphic to a crystal $B(m)_k$, which number of vertices is the multiset coefficients above, where $k=n-s+1=n-d$ and $m=|\lambda|=|\alpha|$. The quasicrystals in $B(\lambda)_n$ are disjoint, and are counted by the standard tableaux of shape $\lambda$. Therefore we have the formula above. Moreover, this formula accounts for the cases where $n< D+1=S$, since the terms of the summation with $n< d+1 \leq D+1$ will be zero.  
\end{proof}

\begin{cor}
	For a partition $\lambda\vDash m$, the number of monic monomials in $s_\lambda(x_1,\ldots, x_n)$ is equal to the sum above.
\end{cor}

Although very interesting, the formula above for the number of tableaux of shpe $\lambda$ and maximal entry $n$ does not solve the much more interesting problem of giving a (closed) formula for Kostka numbers $K^{\lambda}_\mu$, a problem highlighted by Stanley (\cite{Stanley}, Vol.2, section 7.10). Recall that these $K^{\lambda}_\mu$ count the number of tableaux of shape $\lambda\vdash m$ and filling $\mu$ (a fixed composition of $m$).\\ 

There is however a combinatorial formula, given below, for which only a bijective proof is known. The setting of quasicrystals allowes us to give a constructive combinatorial proof, which we discuss below.
\begin{prop}[{\cite[Proposition 5.3.6]{Sagan}}]
Let $\lambda\vdash n$, $S=\{n_1<n_2<\ldots<n_k \}\subseteq [n-1]$, and $\mu=(n_1,n_2-n_1,\ldots,n-n_k)$. Then 
\[|\{P : P \text{ a standard } \lambda\text{-tableau  and } Des(P)\subseteq S\}|=K^\lambda_\mu.\]
\label{prop:SaganKostka}
\end{prop}

\begin{prop}
	A given weight $\mu$ appears exactly once in a quasicrystal $B(\alpha)$, and if and only if $\alpha \preccurlyeq \mu$. 
\end{prop}

\begin{proof}
	We have seen that all weights $\alpha\preccurlyeq \mu$ in at most $n$ parts occur in $B(\alpha)_n$, and only such weights appear in $B(\alpha)_n$. Moreover, for a given minimal parsing of type $\alpha$ of a given shape $\lambda$, $\alpha\preccurlyeq \mu$ defines uniquely a tableau of shape $\lambda$ in the corresponding subcomponent $B(T_\alpha)_n$ in $B(\lambda)_n$, where $T_\alpha$ has the desired minimal parsing. Therefore, in any subcomponent $B(T_\alpha)_n$, there is exactly one tableau of the desired weight $\mu$ in at most $n$ parts.
\end{proof}

By the above, each subcomponent $B(T_\alpha)$ in $B(\lambda)$ with $\alpha\preccurlyeq \mu$ will contain exactly one tableau of filling $\mu$. This gives us the following reinterpretation of Proposition~\ref{prop:SaganKostka}, with a constructive proof.

\begin{cor}
	The Kostka number $K^\lambda_\mu$ counting tableaux of shape $\lambda$ and composition weight $\mu$ is given by the number of  subcomponents $B(T_\alpha)$ of $B(\lambda)$ such that $\alpha\preccurlyeq\mu$, and this number is given by the number of standard tableaux with descent composition $\alpha$ with $\alpha\preccurlyeq \mu$. So
	\[
	K^\lambda_\mu = \left|\{T\in SYT(\lambda) : DesComp(T)\preccurlyeq \mu \} \right|.
	\]
\end{cor}
	
\section{Layout of subcomponents $B(T_\alpha)$ in a crystal $B(\lambda)$}
\label{section:dualposition}

From the previous sections, we have a decomposition of $B(\lambda)$ into a disjoint union of induced subgraphs $B(T_\alpha)$ corresponding to the $F_\alpha$ appearing in the expansion of $s_\lambda$ in the basis of fundamental quasisymmetric functions. Each of these subcomponents regroup tableaux with a specific minimal parsing of type $\alpha$ and correspond to the standard tableau of shape $\lambda$ with the same minimal parsing. \\

In this section, we study how a crystal $B(\lambda)$ breaks down into these subcomponents, by looking at how they are positioned relatively to one another. In order to understand better their distribution, we need to introduce the evacuation involution $\operatorname{EVAC}$ on tableaux, as we'll show it is a crystal anti-automorphism on $B(\lambda)_n$ which reverses descent compositions and respects the above decomposition.\\ 

The evacuation map then sets up a duality between subcomponents associated to descent compositions $\alpha=(\alpha_1,\ldots,\alpha_s)$ and those associated to its reverse composition $\overset{\leftarrow}{\alpha}=(\alpha_s,\ldots,\alpha_1)$. The dual subcomponents will appear in dual position in $B(\lambda)$, and will be dual to one to another as graphs. We also study the (fixed) skeleton structure of $B(\lambda)$ obtained by replacing each subcomponent by the associated standard tableau, and compare this underlying graph structure on standard tableaux with the one of dual equivalence graphs, introduced by Assaf \cite{AssafDualEquivGraph}.

\subsection{Evacuation as a crystal anti-automorphism}

The evacuation involution was first introduced by Schützenberger as an involution on tableaux of the same shape $\lambda$ \cite{SchutzenbergerEvac}.
Berenstein and Zelevinsky showed that the effect of evacuation can be described in the following way \cite{BerensteinZelevinsky}.\\

Let $T$ be a tableau of any shape. Rotate $T$ $180\degres$, change entries $i$ to $n-i+1$, where $n$ is the largest entry in $T$, and rectify this skew tableau using jeu de taquin (see section~\ref{section:RSKPlacticCoplactic} for a recall of the proces of rectifications). We shall use this result as the definition of evacuation: $\operatorname{EVAC}=jdt\circ \text{compl}\circ Rot_{180\degres}$, where $Rot_{180\degres}$, $\text{compl}$ and $jdt$ are the three intermediate manipulations described above. The obtained tableau $\operatorname{EVAC}(T)$ has the same shape as $T$, and applying $\operatorname{EVAC}$ to it recovers $T$ \cite{BerensteinZelevinsky}.

\begin{ex}
	For the straight tableau $T=\young(1123,223,3,4)$, 
	
	\hspace{-9em}
	\begin{minipage}{1.4\textwidth}
		\centering
		$jdt\circ \text{compl}\circ Rot_{180\degres}\left(T\right) = jdt\circ \text{compl} \left(\gyoung(:::;4,:::;3,:;3;2;2,;3;2;1;1)\right) = jdt \left(\gyoung(:::;1,:::;2,:;2;3;3,;2;3;4;4)\right) = \young(1223,234,3,4) = \operatorname{EVAC}\left(T\right)$.
	\end{minipage}
	
	\vspace{1em}
	One may verify that repeating this process on $\operatorname{EVAC}(T)$ recovers $T$.
	
\end{ex}

Berenstein and Zelevinsky showed that $\operatorname{EVAC}$ is an anti-automorphism on crystals $B(\lambda)_n$ \cite{BerensteinZelevinsky}. This means that if a tableau $T$ is obtained from the source $1_\lambda$ of the crystal by a sequence of crystal operators $T=f_{i_k}\circ f_{i_{k-1}}\circ \ldots\circ f_{i_2}\circ f_{i_1}(1_\lambda)$, then $\operatorname{EVAC}(T) = e_{n-i_k}\circ e_{n-i_{k-1}}\circ \ldots\circ e_{n-i_2}\circ e_{n-i_1}(T_{min})$, where $T_{min}$ is the sink of the crystal, which would then be $\operatorname{EVAC}(1_\lambda)$. \\

The original proof of Berenstein and Zelevinsky uses crystal theory and representation theory heavily, so we give an alternate proof of this fact in annex~\ref{annex:proofsEVAC}, which may be more accessible. We also believe our proof to be of interest on its own, as it uses an anti-automorphism of crystals on words, defined by $Rot(w)=w_0 w w_0$, which has been already studied in the literature. We also give a proof that $\operatorname{EVAC}$ reverses descent compositions.

\subsection{Double duality of subcomponents $B(T_\alpha)$ and $B(T_{\overset{\leftarrow}{\alpha}})$ in $B(\lambda)$}

There are two notions of duality at play in this section. The notion of duality which comes from graph (or poset) theory, where the dual of a labelled oriented graph is the graph obtained by reversing arrows, relabelling by $n-i+1$, and replacing each vertex by its "dual image". There is also the notion of duality coming from the involution $\operatorname{EVAC}$, where $\operatorname{EVAC}(A)$ is dual to $A$ for a set of tableaux $A$.\\ 

In the case of the subcomponents of a crystal $B(\lambda)_n$, we will see that these two notions of duality coincide, as evacuation sets up a duality between the subcomponents. For each subcomponent $B(T_\alpha)$, there will be a subcomponent $B(T_{\overset{\leftarrow}{\alpha}})$ such that they are the reciprocal image under the $\operatorname{EVAC}$ map and are the dual graphs of one another, where the dual image of each vertex is precisely its image under the $\operatorname{EVAC}$ map.\\

We also say they are placed dually in $B(\lambda)_n$, 
since $\operatorname{EVAC}$ is an anti-automorphism of crystals on $B(\lambda)_n$. 
It then gives us the following results on the subcomponents of $B(\lambda)$ associated to the fundamental quasisymmetric functions $F_\alpha$. 




\begin{theo*}
	Subcomponents $B(T_\alpha)$ and $B(T_{\overset{\leftarrow}{\alpha}})$ both necessarily appear in a given crystal $B(\lambda)$, they are dual to each other and are positionned in dual locations in $B(\lambda)$. The crystal anti-automorphism between them is the evacuation map $\operatorname{EVAC}$.
\end{theo*}

\begin{proof}
We have seen that the source of a subcomponent $B(T_\alpha)$ is the tableau $T_\alpha$ of shape $\lambda$, with filling and descent composition  $\alpha$. 
%
%
%
Consider a set $A$ of tableaux in $B(T_\alpha)$. Since $\operatorname{EVAC}$ reverses descent compositions, then $\operatorname{EVAC}(A)$ will be a set of tableaux with descent composition $\overset{\leftarrow}{\alpha}$. Moreover, since it is also a crystal anti-automorphism, if a tableau $T\in A$ is obtained from $T_\alpha$ by a sequence of crystal operators $f_i$ which preserve descent compositions, then $\operatorname{EVAC}(T)$ is obtained from $\operatorname{EVAC}(T_\alpha)$ by the complementary sequence of crystal operators $e_{n-i}$, and so the set $\operatorname{EVAC}(A)$ is connected and regroups tableaux with same descent composition $\overset{\leftarrow}{\alpha}$. 
This is because the $f_i$'s modify the weight by $b_{i+1}-b_i$, and the $e_i$'s, by $-b_{i+1}+b_i$, so the effect of $f_i$ on a weight $\gamma$ will be dual to that of $e_{n-i}$ on $\overset{\leftarrow}{\gamma}$.\\ 

Then, subcomponents $B(T_\alpha)$ are mapped onto subcomponents $B(T_{\overset{\leftarrow}{\alpha}})$, and they 
%
are dual to each other as graphs, with the dual image of each vertex given by their image under $\operatorname{EVAC}$.\\

Finally, by the previous argument, we also have that the two subcomponents are positionned dually in $B(\lambda)$: if a sequence of crystal operators $f_{i_k}\circ \ldots \circ f_{i_1}$, applied to the source $1_\lambda$ of $B(\lambda)_n$, gives a vertex of $B(T_\alpha)$, then the "dual" sequence $e_{n-i_k+1}\circ \ldots \circ e_{n-i_1+1}$, applied to the sink $\operatorname{EVAC}(1_\lambda)$ of $B(\lambda)_n$, gives its dual image in $B(T_{\overset{\leftarrow}{\alpha}})$, $\operatorname{EVAC}(T_\alpha)$.
%
\end{proof}

This tells us that the evacuation map respects intrinsically the decomposition of $B(\lambda)$ into its subcomponents $B(T_\alpha)$.

\begin{cor}
The source $T_\alpha$ of a subcomponent  $B(T_\alpha)$ is sent by
$\operatorname{EVAC}$ on the sink of the corresponding dual subcomponent $B(T_{\overset{\leftarrow}{\alpha}})$, and conversely.
\end{cor}


\begin{ex}
Figure~\ref{figure:CrystalDecompInQuasiSymFcts} 
illustrates the duality (and symmetry) of the positioning of the subcomponents $B(T_\alpha)$ and $B(T_{\overset{\leftarrow}{\alpha}})$. Note how the two subcomponents associated to $(2,3,2)$ are dual, and start and end on the same rows of the crystal.
\label{ex:DualitySymmetricComp}
\end{ex}


\begin{prop}
	The quasicrystals $B(\alpha)$ are self-dual as labelled oriented graphs when $\alpha$ is symmetric, i.e. $\alpha= \overset{\leftarrow}{\alpha}$.
	\label{prop:symetricAlpha}
\end{prop}

\begin{rem}
	This self-duality holds on the structure of quasicrystals $B(\alpha)$ (as graphs). However, the subcomponents $B(T_\alpha)$ are not necessarily sent onto themself under the $\operatorname{EVAC}$ map, as illustrated in example~\ref{ex:DualitySymmetricComp}.
\end{rem}

\begin{proof}[Proof of Proposition~\ref{prop:symetricAlpha}]
	$\operatorname{EVAC}$ is an anti-automorphism which sends one subcomponent $B(T_\alpha)$ onto a subcomponent $B(T_{\overset{\leftarrow}{\alpha}})=B(T_\alpha')$. Since both subcomponents are then isomorphic and dual, and all subcomponents in $B(\alpha)$ are isomorphic, then $B(\alpha)$ is auto-dual as a labelled oriented graph when $\alpha = \overset{\leftarrow}{\alpha}$.
\end{proof}

\begin{cor}
For a fixed composition $\alpha$, all subcomponents $B(T_\alpha)$ have their source on the same row $j+1$ of $B(\lambda)$, where $j$ is the number of transformations $+(b_{i+1}-b_{i})=+(0,\ldots,0,-1,1,0,\ldots,0)$ applied to $\lambda$ to obtain $\alpha$, or equivalently the number $j$ of crystal operators $f_i$ applied to $1_\lambda$ to obtain $T_\alpha$.
\end{cor}

\begin{proof}
We have seen that the crystal operators $f_i$ have effect $+(b_i-b_{i-1})$ on the weight of tableaux. Then row $j+1$ of $B(\lambda)$ holds all tableaux obtained from $1_\lambda$ by applying $j$ crystal operators $f_i$. If a certain sequence of crystal operators $f_{i_1}, f_{i_2}, \ldots , f_{i_j}$ modify $\lambda$ to obtain $\alpha$, then all tableaux of weight $\alpha$ are obtained by the application of a certain reordering of these crystal operators. Then, all tableaux of weight $\alpha$ will be on the same row $j+1$. Among these, all tableaux with weight and minimal parsing of type $\alpha$ lie on this row $j+1$ of $B(\lambda)$. Since these are the sources of the subcomponents of $B(\alpha)$, then we have the desired result.
\end{proof}

\begin{rem}
	In a crystal $B(\lambda)$, all descent compositions $\alpha$ occur as weights of tableaux, and are therefore obtained from $\lambda$ by applying modifications $b_{i+1}-b_i$, from $\lambda$ to $\overset{\leftarrow}{\lambda}$. Moreover, all the descent compositions $\alpha$ are not refinements of another and do not include zero parts (by their definition as counting lengths of horizontal bands in minimal parsings).
\end{rem}

\subsection{Crystal skeleton}

If we replace every subcomponent of $B(\lambda)_n$ by the associated standard tableaux of shape $\lambda$, and keep only one oriented edge between linked subcomponents, with one copy of all labels appearing at least once, one gets a labelled oriented graph on standard tableaux. Note that this can create cycles. We will see that we can restrict this further by keeping only the minimal label on each oriented edge. We call the result the skeleton of $B(\lambda)_n$.\\

This notion of skeleton is especially interesting, because it gives a compact visual representation of $B(\lambda)_n$ for any $n$, 
and also of $B(\lambda)$, as we will see.\\ 

Doing this to figure~\ref{figure:CrystalDecompInQuasiSymFcts}, one gets the labelled oriented graph on standard tableaux of shape $(4,3)$ of figure~\ref{fig:graphTabStd}.
%
Since we know the oriented graph structure of these subcomponents by corollary~\ref{prop:GraphStructureSubcomponents}, these can be expanded to essentially retreive the full graph, with some edges missing between subcomponents. Note the symmetry coming from the auto-duality of $B(\lambda)$. \\ 

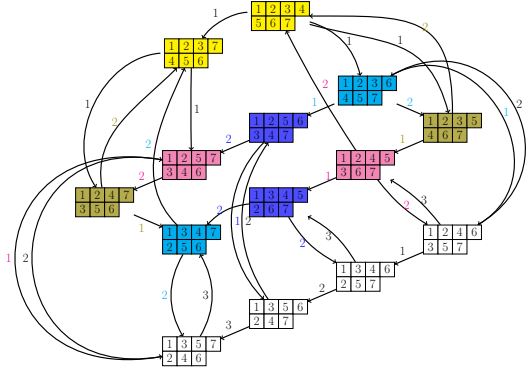
\begin{figure}[h!]
	\Yboxdim{0.7cm}
	\begin{center}
		\resizebox{0.5\textwidth}{!}{
			\begin{tikzpicture}[scale=0.6]
				\tikzset{myptr/.style={decoration={markings,mark=at position 1 with {\arrow[scale=3,>=stealth]{>}}},postaction={decorate}}}
				\tikzstyle{every node}=[font=\Large]
				\begin{pgfonlayer}{nodelayer}
					\node [style=none] (0) at (0, 8) {\Yfillcolour{yellow} \young(1234,567)};
					\node [style=none] (1) at (7, 2) {\Yfillcolour{cyan} \young(1236,457)};
					\node [style=none] (2) at (-7, 5) {\Yfillcolour{yellow} \young(1237,456)};
					\node [style=none] (3) at (0, -1) {\Yfillcolour{blue!70}\young(1256,347)};
					\node [style=none] (4) at (14, -1) {\Yfillcolour{olive!70}\young(1235,467)};
					\node [style=none] (5) at (7, -4) {\Yfillcolour{magenta!60}\young(1245,367)};
					\node [style=none] (6) at (-7, -4) {\Yfillcolour{magenta!60}\young(1257,346)};
					\node [style=none] (7) at (-14, -7) {\Yfillcolour{olive!70}\young(1247,356)};
					\node [style=none] (8) at (0, -7) {\Yfillcolour{blue!70}\young(1345,267)};
					\node [style=none] (9) at (-7, -10) {\Yfillcolour{cyan}\young(1347,256)};
					\node [style=none] (10) at (14, -10) {\young(1246,357)};
					\node [style=none] (11) at (7, -13) {\young(1346,257)};
					\node [style=none] (12) at (0, -16) {\young(1356,247)};
					\node [style=none] (13) at (-7, -19) {\young(1357,246)};
				\end{pgfonlayer}
				\begin{pgfonlayer}{edgelayer}
					\draw [->, ultra thick, bend left] (0) to node[right]{1} (1);
					\draw [->, ultra thick, in=105, out=-15, looseness=1.25] (0) to node[above]{1} (4);
					\draw [->, ultra thick] (1) to node[above]{\textcolor{cyan}{2}} (4);
					\draw [->, ultra thick, bend right, looseness=0.75] (0) to node[above left]{1} (2);
					\draw [->, ultra thick, bend left=3] (5) to node[above right]{\textcolor{magenta}{2}} (0);
					\draw [->, ultra thick, in=360, out=90, looseness=1.50] (4) to node[right]{\textcolor{olive}{2}} (0);
					\draw [->, ultra thick, bend right=1] (2) to node[right]{1} (6);
					\draw [->, ultra thick, bend right=60] (2) to node[left]{1} (7);
					\draw [->, ultra thick, in=-135, out=100, looseness=1] (7) to node[left]{\textcolor{olive}{2}} (2);
					\draw [->, ultra thick, in=-115, out=135] (9) to node[above left]{\textcolor{cyan}{2}} (2);
					\draw [->, ultra thick] (1) to node[above left]{\textcolor{cyan}{1}} (3);
					\draw [->, ultra thick, bend left=90, looseness=1.5] (1) to node[right]{\textcolor{cyan}{1}} (10);
					\draw [->, ultra thick, bend right=90, looseness=1.8] (10) to node[right]{2} (1);
					\draw [->, ultra thick] (4) to node[above left]{\textcolor{olive}{1}} (5);
					\draw [->, ultra thick] (3) to node[above left ]{\textcolor{blue}{2}} (6);
					\draw [->, ultra thick, bend right=45, looseness=1] (3) to node[right]{\textcolor{blue}{1}} (12);
					\draw [->, ultra thick, bend left=35, looseness=1] (12) to node[right]{2} (3);
					\draw [->, ultra thick] (5) to node[above right]{\textcolor{magenta}{1}} (8);
					\draw [->, ultra thick, bend right=10] (5) to node[right]{\textcolor{magenta}{2}} (10);
					\draw [->, ultra thick, bend right=10] (10) to node[right]{3} (5);
					\draw [->, ultra thick] (6) to node[above left]{\textcolor{magenta}{2}} (7);
					\draw [->, ultra thick, bend right=100, looseness=2.6] (6) to node[left]{\textcolor{magenta}{1}} (13);
					\draw [->, ultra thick, bend left=100, looseness=2.3] (13) to node[left]{2} (6);
					\draw [->, ultra thick] (7) to node[below left]{\textcolor{olive}{1}} (9);
					\draw [->, ultra thick, bend right=20] (8) to node[left]{\textcolor{blue}{2}} (9);
					\draw [->, ultra thick, bend right=15] (8) to node[left]{\textcolor{blue}{2}} (11);
					\draw [->, ultra thick, bend right=15] (11) to node[left]{3} (8);
					\draw [->, ultra thick, bend right] (9) to node[left]{\textcolor{cyan}{2}} (13);
					\draw [->, ultra thick, bend right] (13) to node[left]{3} (9);
					\draw [->, ultra thick] (10) to node[above left]{1} (11);
					\draw [->, ultra thick] (11) to node[above]{2} (12);
					\draw [->, ultra thick] (12) to node[above left]{3} (13);
				\end{pgfonlayer}
		\end{tikzpicture}}
	\end{center}
	\caption{Skeleton of $B(4,3)_4$, where subcomponents are replaced by their associated standard tableau. Labeled oriented edges $i$ indicate that transformations $f_i$ move from one subcomponent to another, and $i$ is minimal (by definition). Labels are colored according to the origin subcomponent for clarity. The vertical position of a vertex is determined by the row of the source of the associated connected component in $B(4,3)_4$, and alternatively also segregates standard tableaux by their number of descents. \label{fig:graphTabStd}}
\end{figure}

The obtained labelled oriented graph on standard tableaux does not have a crystal structure, in particular because a standard tableau can have two distinct images for the same transformations. For example, in figure~\ref{fig:graphTabStd}, the top standard tableau (sometimes refered to as the \textit{superstandard tableau}) has three distinct images under the transformations $1$. We will see in the proposition below that we are justified in keeping only the minimal label of oriented edges between subcomponents. 

\begin{prop}
	Let $n$ and $\lambda = (\lambda_1,\lambda_2, \ldots,\lambda_\ell)$ be fixed. Let $std(T_\alpha)$ and $std(T_\beta)$ be standard tableaux of shape $\lambda$, and of respective descent compositions $\alpha=(\alpha_1,\ldots,\alpha_k)$ and $\beta=(\beta_1,\ldots,\beta_s)$. Let $i$ be the smallest index such that a crystal operator $f_i$ allows one to pass between the associated subcomponents $B(T_\alpha)$ and $B(T_\beta)$ in $B(\lambda)_n$, for $1\leq i\leq n-1$. Then the crystal operators $f_{i+1},\ldots,f_{i+(n-s)}$ do too.
\end{prop}

\begin{proof}
	Crystal operators act on weights $\gamma$ of tableaux in $B(\lambda)_n$. If a tableau $T$ lies in $B(T_\alpha)$, it has also $\alpha \preccurlyeq wt(T)=\gamma$. If $f_i(T)$ lies in $B(T_\beta)$, then it means that $\beta\preccurlyeq  wt(f_i(T))=\gamma+(b_{i+1}-b_i)$.
	This will also be the case for $(0^k,\gamma)$, for $0\leq k\leq n-s$, with the crystal operator having the corresponding actions on the weights being $f_{i+k}$.
\end{proof}

This means we can keep only the smallest label $i$ of oriented edges between subcomponents. We then call the obtained skeleton 
$Skeleton(\lambda)_n$.

\begin{theo*}
	For $\lambda\vdash m$ fixed, and let $S$ be the maximal length of descent compositions for $\lambda$. Then the skeletons $Skeleton(\lambda)_n$ of the crystals $B(\lambda)_n$ are equal for all $n\geq S$.
	For $1\leq n\leq S$, the skeleton of $B(\lambda)_n$ consists of the induced subgraph of $Skeleton(\lambda)_S$ containing standard tableaux of shape $\lambda$ with descent compositions in at most $n$ parts. 
	\label{prop:SkeletonS}
\end{theo*}

\begin{proof}
	Let's first consider two random subcomponents $B(T_\alpha)$, $B(T_\beta)$, for $T_\alpha,T_\beta$ of the same shape $\lambda$, with respective weight and minimal parsing of type $\alpha$ and $\beta$, for incomparable descent compositions $\alpha,\beta$. Both subcomponents occur in all $B(\lambda)_n$ for $n\geq N=\max(\ell(\alpha),\ell(\beta))$.\\ 
	
	Suppose that there exists a minimal value $k\in\N$ such that there exists an edge labelled $i$ from $B(T_\alpha)$ into $B(T_\beta)$ in $B(\lambda)_{N+k}$. There are then two tableaux $T_{\gamma^{(1)}}\in B(T_\alpha)$ and $T_{\gamma^{(2)}}\in B(T_\beta)$ such that $\alpha\preccurlyeq \gamma^{(1)}$, $\beta\preccurlyeq \gamma^{(2)}$ and $f_i(T_{\gamma^{(1)}})=T_{\gamma^{(2)}}$.\\

	We then have
	\begin{center}
	$\alpha\preccurlyeq \gamma^{(1)} = (\gamma_1,\gamma_2,\ldots,\gamma_{i-1},\gamma_i,\gamma_{i+1},\gamma_{i+1},\ldots,\gamma_{N+k})$ \\
	$\beta\preccurlyeq \gamma^{(2)} = (\gamma_1,\gamma_2,\ldots,\gamma_{i-1},\gamma_i-1,\gamma_{i+1}+1,\gamma_{i+1},\ldots,\gamma_{N+k})$.
	\end{center}
	
	The crystal operator $f_i$ changes exactly one entry $i$ into an entry $i+1$ in $T_{\gamma^{(1)}}$, but changes the minimal parsing. Therefore, it must act exactly on two maximal horizontal bands of $T_{\gamma^{(1)}}$, modifying the maximal horizontal band containing that entry $i$ and the one containing $(i+1)$'s in $T_{\gamma^{(1)}}$. Moreover, the entries $i$ and $i+1$ cannot be part of the same maximal horizontal band, otherwise the minimal parsing would not be changed. \\
	
	Recall that we can express a decomposition of $\gamma^{(1)}$ and $\gamma^{(2)}$ into subsets of parts, respectively summing to the parts of $\alpha$ and $\beta$, to represent the minimal parsing respectively of $T_{\gamma^{(1)}}$ and $T_{\gamma^{(2)}}$. Since they differ in exactly two maximal horizontal bands, then the decompositions of $\gamma^{(1)}$ and $\gamma^{(2)}$ have all subsets equal, except for those containing the $i^{\text{th}}$ and $(i+1)^{\text{th}}$ parts. Suppose then that $j_1<j_2<\ldots$ give the position of the last part of each subset, this gives the following decompositions, with potentially an added separation $|$ before the $i^{\text{th}}$ part and/or after the $(i+1)^{\text{th}}$ part, in $\gamma^{(1)}$ and/or $\gamma^{(2)}$.
	
	\begin{center}
	$\alpha\preccurlyeq \gamma^{(1)} = (\gamma_1,\ldots,\gamma_{j_1}|\gamma_{j_1+1},\ldots,\gamma_{j_2}|\ldots|\gamma_{j_\ell+1},\ldots,\gamma_i|\gamma_{i+1},\ldots, \gamma_{\ell+1}|\ldots, \gamma_{N+k})$ \\
	$\beta\preccurlyeq \gamma^{(2)} = (\gamma_1,\ldots,\gamma_{j_1}|\gamma_{j_1+1},\ldots,\gamma_{j_2}|\ldots|\gamma_{j_\ell+1},\ldots,\gamma_i-1|\gamma_{i+1}+1,\ldots, \gamma_{\ell+1}|\ldots, \gamma_{N+k})$.
	\end{center}

	Now, since $k$ is minimal by hypothesis, then there cannot be equal subsets with more than one part, because otherwise there exists two tableaux $T_{{\gamma^{(1)}}'}, T_{{\gamma^{(2)}}'}$ respectively in $B(T_\alpha)$ and $B(T_\beta)$ which have weights of smaller length that $N+k$ obtained by summing parts of equal subsets, with $T_{{\gamma^{(1)}}'} \xrightarrow{j} T_{{\gamma^{(2)}}'}$ in $B(\lambda)_n$ for $n<N+k$ and $j\leq i$.\\
	
	For the same reason, there cannot be more than two parts in the subsets containing the $i^{\text{th}}$ and $(i+1)^{\text{th}}$ parts in $\gamma^{(1)}$ and $\gamma^{(2)}$.\\
	
	We then have that the decomposition above is coarser, with again potentially an additionnal separation $|$ before the $i^{\text{th}}$ part and/or after the $(i+1)^{\text{th}}$ part, in $\gamma^{(1)}$ and/or $\gamma^{(2)}$:
	\begin{center}
		$\alpha\preccurlyeq \gamma^{(1)} = (\gamma_1|\gamma_2|\ldots|\gamma_{i-1},\gamma_i|\gamma_{i+1},\gamma_{i+2}|\ldots|\gamma_{N+k})$ \\
		$\beta\preccurlyeq \gamma^{(2)} = (\gamma_1|\gamma_2|\ldots|\gamma_{i-1},\gamma_i-1|\gamma_{i+1}+1,\gamma_{i+2}|\ldots|\gamma_{N+k})$.
	\end{center}

	The different decompositions into subsets of the parts in posititon $i-1$ to $i+1$ are given below for $\gamma^{(1)}$, the ones for $\gamma^{(2)}$ are equivalent.
	
	\begin{center}
	$\gamma_{i-1}|\gamma_i|\gamma_{i+1}|\gamma_{i+2}$ \quad
	 $\gamma_{i-1},\gamma_i|\gamma_{i+1},\gamma_{i+2}$ \quad
	   $\gamma_{i-1},\gamma_i|\gamma_{i+1}|\gamma_{i+2}$ \quad $\gamma_{i-1}|\gamma_i|\gamma_{i+1},\gamma_{i+2}$.
	\end{center}	

	Therefore $k\leq 2$.\\
	
	Let's now study the different possibilities, depending on the position of the cell $i$ modified by $f_i$, in relationship to the (non-maximal) horizontal bands containing the entries $i-1,i,i+1$ and $i+2$.\\
	
	As we have seen, the entries $i$ and $i+1$ have to be in different maximal horizontal bands, otherwise the minimal parsing is preserved by $f_i$. Then the $i$'s appear at the end of their maximal horizontal band, and the $(i+1)$'s, at the start of theirs. We call the specific entry $i$ modified by $f_i$ in $T_{\gamma^{(1)}}$ \textit{the modified entry $i$}. All entries considered are in $T_{\gamma^{(1)}}$. We consider how the modification of one entry $i$ changes the division of $\gamma^{(1)}$ into subsets to get that of $\gamma^{(2)}$.\\
	
	Recall that the head of a (non-maximal) horizontal band is its northeastmost cell, and its tail, its southwestmost cell.\\

	\noindent \underline{Case 1 :} If the modified entry $i$ lays on a row of index strictly smaller than that of the tail of the horizontal band of the $i+1$'s, and weakly greater than that of the head of the horizontal band of the $i+1$'s, then the divisions in the corresponding weights $\gamma^{(1)}$ and $\gamma^{(2)}$ are in the same positions, so we say the divisions are preserved. This is because the change of that single entry $i$ does not interfere with the entries $i-1$ or $i+2$.\\
	
	\noindent \underline{Case 2 :} If the modified entry $i$ lays on a row of index weakly greater than that of the tail of the horizontal band of the $i+1$'s, then there are three cases to consider. 
	\begin{itemize}[itemsep=0ex]
		\item If the modified entry $i$ is not the tail of the horizontal band of entries $i$, then the divisions are preserved.
		\item If the modified entry $i$ is the tail of the horizontal band of entries $i$, and the next entry $i$ of the horizontal band is southwest of the head of the horizontal band of the $i-1$'s, then the divisions are preserved.
		\item If the modified entry $i$ is the tail of the horizontal band of entries $i$, and the next entry $i$ of the horizontal band is weakly northeast of the head of the horizontal band of the $i-1$'s, then if there is a division between the $(i-1)^{\text{th}}$ and $i^{\text{th}}$ parts in $\gamma^{(1)}$, then it is removed in $\gamma^{(2)}$. All other divisions are preserved. 
	\end{itemize}

	\noindent \underline{Case 3 :} If the modified entry $i$ is weakly northeast of the head of the horizontal band of the $i+1$'s, then there are similarly three cases to consider. 
	\begin{itemize}[itemsep=0ex]
		\item If the modified entry $i$ is not the head of the horizontal band of entries $i$, then the divisions are preserved.
		\item If the modified entry $i$ is the head of the horizontal band of entries $i$, and the tail of the horizontal band of the $i+2$'s is to its northwest, then the divisions are preserved.
		\item If the modified entry $i$ is the head of the horizontal band of entries $i$, and the tail of the horizontal band of the $i+2$'s is weakly to its southwest, then if there is no division between the $(i+1)^{\text{th}}$ and $(i+2)^{\text{th}}$ parts in $\gamma^{(1)}$, then it is added in $\gamma^{(2)}$. All other divisions are preserved. 
	\end{itemize}

	There are then very limited cases when a division is either added or removed, and otherwise divisions are preserved. Let's then consider what transitions are possible from the possible configurations of $\gamma^{(1)}$.\\
	
	Let's start with the configuration $(\gamma_1|\ldots|\gamma_{i-1}|\gamma_i|\gamma_{i+1}|\gamma_{i+2}|\ldots|\gamma_{N+k})$. Either divisions are preserved, or the division between the $(i-1)^{\text{th}}$ and $i^{\text{th}}$ parts is removed, to get either configurations below in $\gamma^{(2)}$.
	
	\begin{center}
		$(\gamma_1|\ldots|\gamma_{i-1}|\gamma_i-1|\gamma_{i+1}+1|\gamma_{i+2}|\ldots|\gamma_{N+k})$ OR $(\gamma_1|\ldots|\gamma_{i-1},\gamma_i-1|\gamma_{i+1}+1|\gamma_{i+2}|\ldots|\gamma_{N+k})$.
	\end{center}

	Let's now consider the configuration $(\gamma_1|\ldots|\gamma_{i-1},\gamma_i|\gamma_{i+1},\gamma_{i+2}|\ldots|\gamma_{N+k})$. Either divisions are preserved, or the division between the $(i+1)^{\text{th}}$ and $(i+2)^{\text{th}}$ parts can be added, to get either configurations below in $\gamma^{(2)}$.
	
	\begin{center}
		$(\gamma_1|\ldots|\gamma_{i-1},\gamma_i-1|\gamma_{i+1}+1,\gamma_{i+2}|\ldots|\gamma_{N+k})$ OR $(\gamma_1|\ldots|\gamma_{i-1},\gamma_i-1|\gamma_{i+1}+1|\gamma_{i+2}|\ldots|\gamma_{N+k})$.
	\end{center}
	However, in these two cases, the $(i-1)^{\text{th}}$ and $i^{\text{th}}$ parts can be summed (in $\gamma^{(1)}$ and $\gamma^{(2)}$) to retreive valid weights of smaller length, so they must be rejected.\\
	
	Let's now consider the configuration $(\gamma_1|\ldots|\gamma_{i-1},\gamma_i|\gamma_{i+1}|\gamma_{i+2}|\ldots|\gamma_{N+k})$. Divisions can only be preserved here, to get the configuration below in $\gamma^{(2)}$.
	
	\begin{center}
		$(\gamma_1|\ldots|\gamma_{i-1},\gamma_i-1|\gamma_{i+1}+1|\gamma_{i+2}|\ldots|\gamma_{N+k})$.
	\end{center}

	Similarly as for the previous configuration, the $(i-1)^{\text{th}}$ and $i^{\text{th}}$ parts can be added (in $\gamma^{(1)}$ and $\gamma^{(2)}$) to retreive weights of smaller length, so this must be rejected.\\
	
	Let's finally consider the configuration $(\gamma_1|\ldots|\gamma_{i-1}|\gamma_i|\gamma_{i+1},\gamma_{i+2}|\ldots|\gamma_{N+k})$. Either divisions are preserved, or the division between the $(i+1)^{\text{th}}$ and $(i+2)^{\text{th}}$ parts is added, to get either configurations below in $\gamma^{(2)}$.
	
	\begin{center}
		$(\gamma_1|\ldots|\gamma_{i-1}|\gamma_i-1|\gamma_{i+1}+1,\gamma_{i+2}|\ldots|\gamma_{N+k})$ OR $(\gamma_1|\ldots|\gamma_{i-1}|\gamma_i-1|\gamma_{i+1}+1|\gamma_{i+2}|\ldots|\gamma_{N+k})$.
	\end{center}

	In the first configuration above, the $(i+1)^{\text{th}}$ and $(1+2)^{\text{th}}$ parts can be added (in $\gamma^{(1)}$ and $\gamma^{(2)}$) to retreive weights of smaller length, so this configuration must be rejected. The second one is valid. \\
	
	There are then only three possible configurations for the divisions in $\gamma^{(1)} \xrightarrow{i} \gamma^{(2)}$, such that $k$ is minimal and the minimal parsing is modified by $f_i$, going from $T_{\gamma^{(1)}}$ to $T_{\gamma^{(2)}}$. What is most important for us here is to note that they force $k=0$, and that $|\ell(\alpha)-\ell(\beta)|\in\{0,1\}$.\\
	
	Note that there is a special case to consider in cases where $\gamma_i=1$, with $\gamma_{i+1}\geq 1$ or $\gamma_{i+1}=0$. The second case is easy since then the maximal horizontal bands are preserved. In the first case, the only possible configurations are the following.
	
	\begin{center}
		$(\gamma_1|\ldots|\gamma_{i-1}|1|\gamma_{i+1},\gamma_{i+2}|\ldots|\gamma_{N+k}) \xrightarrow{i} (\gamma_1|\ldots|\gamma_{i-1},0|\gamma_{i+1}+1|\gamma_{i+2}|\ldots|\gamma_{N+k})$ OR $(\gamma_1|\ldots|\gamma_{i-1}|1|\gamma_{i+1}|\gamma_{i+2}|\ldots|\gamma_{N+k}) \xrightarrow{i} (\gamma_1|\ldots|\gamma_{i-1},0|\gamma_{i+1}+1|\gamma_{i+2}|\ldots|\gamma_{N+k})$.
	\end{center}

	In the second configuration, we still get $k=0$. In the first one, these configurations are only possible if the only entry $i$ is northeast of both the head of the horizontal band of the $i+1$'s and the tail of the $i+2$'s (in order to break their maximal horizontal band). However, $f_i$ cannot change this entry, since the corresponding parenthesis sequence (for the parenthesis rule) will give $(\ldots()$, with at least one parenthesis $($ to be paired with the parenthesis $)$ of the entry $i$, so $f_i$ must be null and the first configurations cannot occur. Then, in all cases, $k=0$.\\

	Therefore, if there exist a minimal edge labelled $i$ between two subcomponents $B(T_\alpha)$ and $B(T_\beta)$, it occurs in $B(\lambda)_N$ for $N=\max(\ell(\alpha),\ell(\beta))$, and in all $B(\lambda)_n$ for $n\geq N$. Moreover, edges only occur between subcomponents associates to descent compositions which have equal length or which lengths differ only by $1$.\\

	Now, $B(\lambda)_S$ contains all tableaux of shape $\lambda$ and filling at most $S$. Since $S$ is the maximal number of parts in descent compositions for $\lambda$, then all subcomponents associated to quasisymmetric functions $F_\alpha$ occur in $B(\lambda)_S$, with at least one tableau (if $\ell(\alpha)=S$). By the above result, all minimal edges will then also occur in $B(\lambda)_S$.\\
	
	For all $n\geq S$, all subcomponents occur, potentially with more tableaux.
	For $n<S$, then some subcomponents will be missing, but the minimal edges between occuring subcomponents will also occur by the previous result, so the obtained skeleton is the induced subgraph of $Skeleton(\lambda)_S$ containing the standard tableaux with descent composition of length at most $n$ as vertices. This gives the wanted result.

\end{proof}

We can define $Skeleton(\lambda) = Skeleton(\lambda)_S$, where $S$ is the maximal length of a descent composition for $\lambda$. Then $Skeleton(\lambda)$ is also the underlying structure of $B(\lambda)$.\\
It is a corollary of the proof that
\begin{cor}
	There are edges in $Skeleton(\lambda)$ only between standard tableaux whose number of descents differ by at most $1$.
\end{cor}

\begin{conj}
	Let $H_S$ be the induced subgraphs of $Skeleton(\lambda)$ whose vertices are the standard tableaux with descent compositions having $s$ parts. Then $H_s$ is either a
	\begin{itemize}[itemsep=0ex]
		\item Disjoint union of singleton(s), or
		\item Disjoint union of chain(s), or
		\item Disjoint union of even cyle(s) with, or without, two extra attached vertices giving the source(s) and sink(s). 
	\end{itemize}
	Multiple edges occur only between such induced subgraphs associated to different descent composition lengths $s$.
	\label{conj:structureSubGraphsSkeleton} 
\end{conj}

This has been verified for all partitions $\lambda\vdash m$ with $m\leq6$. Figure~\ref{fig:evenCycles} illustrates different cases of the conjecture.\\ 

\begin{figure}
	\begin{subfigure}[b]{0.47\textwidth}
		\Yboxdim{0.7cm}
		\centering
			\resizebox{0.44\textwidth}{!}{
				\begin{tikzpicture}[scale=0.6]
				\tikzset{myptr/.style={decoration={markings,mark=at position 1 with {\arrow[scale=3,>=stealth]{>}}},postaction={decorate}}}
				\tikzstyle{every node}=[font=\Large]
				\begin{pgfonlayer}{nodelayer}
				\node [style=none] (0) at (0, 15) {\Yfillcolour{yellow} \young(1234,5,6)};
				\node [style=none] (1) at (0, 10) {\Yfillcolour{cyan} \young(1235,4,6)};
				\node [style=none] (2) at (-5, 5) {\Yfillcolour{magenta} \young(1236,4,5)};
				\node [style=none] (3) at (5, 5) {\Yfillcolour{blue!70}\young(1245,3,6)};
				\node [style=none] (4) at (0, 0) {\Yfillcolour{olive!70}\young(1246,3,5)};
				\node [style=none] (5) at (10, 0) {\Yfillcolour{magenta!50}\young(1345,2,6)};
				\node [style=none] (6) at (-5, -5) {\Yfillcolour{magenta}\young(1256,3,4)};
				\node [style=none] (7) at (5, -5) {\Yfillcolour{blue!70}\young(1346,2,5)};
				\node [style=none] (8) at (0, -10) {\Yfillcolour{cyan}\young(1356,2,4)};
				\node [style=none] (9) at (0, -15) {\Yfillcolour{yellow}\young(1456,2,3)};
				\end{pgfonlayer}
				\begin{pgfonlayer}{edgelayer}
				\draw [->, ultra thick] (0) to node[left]{1} (1);
				\draw [->, ultra thick] (1) to node[above left]{2} (2);
				\draw [->, ultra thick] (1) to node[above right]{1} (3);
				\draw [->, ultra thick] (2) to node[left]{1} (4);
				\draw [->, ultra thick] (3) to node[below right]{2} (4);
				\draw [->, ultra thick] (3) to node[right]{1} (5);
				\draw [->, ultra thick] (4) to node[right]{2} (6);
				\draw [->, ultra thick] (4) to node[above right]{1} (7);
				\draw [->, ultra thick] (5) to node[below right ]{2} (7);
				\draw [->, ultra thick] (6) to node[above right]{1} (8);
				\draw [->, ultra thick] (7) to node[above right]{2} (8);
				\draw [->, ultra thick] (8) to node[right]{2} (9);
				\end{pgfonlayer}
				\end{tikzpicture}}
	\vspace{1.5em}
	\caption{$Skeleton(4,1,1)$, with all tableaux having two descents and forming a union of even cycles with extra vertices giving the source and sink}
	\end{subfigure} \hfill
	\begin{subfigure}[b]{0.51\textwidth}
			\Yboxdim{0.7cm}
			\centering
				\resizebox{0.99\textwidth}{!}{
					\begin{tikzpicture}[scale=0.6]
					\tikzset{myptr/.style={decoration={markings,mark=at position 1 with {\arrow[scale=3,>=stealth]{>}}},postaction={decorate}}}
					\tikzstyle{every node}=[font=\Large]
					\begin{pgfonlayer}{nodelayer}
					\node [style=none] (0) at (0, 20) {\Yfillcolour{yellow} \young(123,45,6)};
					\node [style=none] (1) at (-5, 15) {\Yfillcolour{cyan} \young(123,46,5)};
					\node [style=none] (2) at (5, 15) {\Yfillcolour{magenta} \young(125,34,6)};
					\node [style=none] (3) at (-10, 10) {\Yfillcolour{orange}\young(124,36,5)};
					\node [style=none] (4) at (10, 10) {\Yfillcolour{orange}\young(126,34,5)};
					\node [style=none] (5) at (-5, 5) {\Yfillcolour{magenta}\young(134,26,5)};
					\node [style=none] (6) at (5, 5) {\Yfillcolour{cyan}\young(126,35,4)};
					\node [style=none] (7) at (15, 5) {\Yfillcolour{olive!60}\young(124,35,6)};
					\node [style=none] (8) at (-15, 0) {\Yfillcolour{black!30}\young(125,36,4)};
					\node [style=none] (9) at (0, 0) {\Yfillcolour{yellow}\young(136,25,4)};
					\node [style=none] (10) at (10, 0) {\Yfillcolour{magenta!40}\young(134,25,6)};
					\node [style=none] (11) at (-10, -5) {\Yfillcolour{blue!60}\young(135,26,4)};
					\node [style=none] (12) at (5, -5) {\Yfillcolour{blue!60}\young(135,24,6)};
					\node [style=none] (13) at (-5, -10) {\Yfillcolour{magenta!40}\young(145,26,3)};
					\node [style=none] (14) at (0, -10) {\Yfillcolour{black!30}\young(136,24,5)};
					\node [style=none] (15) at (0, -15) {\Yfillcolour{olive!60}\young(146,25,3)};
					\end{pgfonlayer}
					\begin{pgfonlayer}{edgelayer}
					\draw [->, ultra thick] (0) to node[above left]{2} (1); 
					\draw [->, ultra thick] (0) to node[above right]{1} (2);
					\draw [->, ultra thick] (1) to node[above left]{1} (3);
					\draw [->, ultra thick] (2) to node[below left]{2} (4);
					\draw [->, ultra thick] (3) to node[above right]{2} (5);
					\draw [->, ultra thick] (4) to node[above left]{2} (6);
					\draw [->, ultra thick] (5) to node[above right]{2} (9);
					\draw [->, ultra thick] (6) to node[above left]{1} (9);
					
					\draw [->, ultra thick] (7) to node[above left]{1} (10); 
					\draw [->, ultra thick] (10) to node[above left]{2} (12);
					\draw [->, ultra thick] (12) to node[above left]{3} (14);
					
					\draw [->, ultra thick] (8) to node[below left]{1} (11); 
					\draw [->, ultra thick] (11) to node[below left]{2} (13);
					\draw [->, ultra thick] (13) to node[below left]{3} (15);
					
					\draw [->, ultra thick, bend left=40, looseness=1.5] (0) to  node[right]{1} (7); 
					\draw [->, ultra thick, bend right=50, looseness=1.6] (7) to node[right]{2} (0);
					\draw [->, ultra thick, bend left=70, looseness=2.1] (2) to node[right]{1} (12);
					\draw [->, ultra thick, bend right=80, looseness=2.2] (12) to node[right]{2} (2);
					\draw [->, ultra thick, in=-5, out=-10, looseness=2.2] (4) to node[right]{1} (14);
					\draw [->, ultra thick, in=-5, out=-10, looseness=2.4] (14) to node[right]{2} (4);
					\draw [->, ultra thick, bend right=60, looseness=1.6] (3) to node[left]{2} (8);
					\draw [->, ultra thick, bend left=40, looseness=1.5] (8) to node[left]{3} (3);
					\draw [->, ultra thick, bend right=40, looseness=1.6] (5) to node[left]{2} (11);
					\draw [->, ultra thick, bend left=20,looseness=1.5] (11) to node[above left]{3} (5);
					\draw [->, ultra thick, bend right=35, looseness=1.1] (9) to node[right]{2} (15);
					\draw [->, ultra thick, bend left=25, looseness=1] (15) to node[right]{3} (9);
					\end{pgfonlayer}
					\end{tikzpicture}}
	\caption{$Skeleton(3,2,1)$, with standard tableaux with two descents forming an even cycle, and those with three descents forming a disjoint union of chains}	
	\end{subfigure}
	\caption{Structure of induced subgraphs of $Skeleton(\lambda)$ whose vertices are standard tableaux with a fixed numer of descents \label{fig:evenCycles}}
\end{figure}
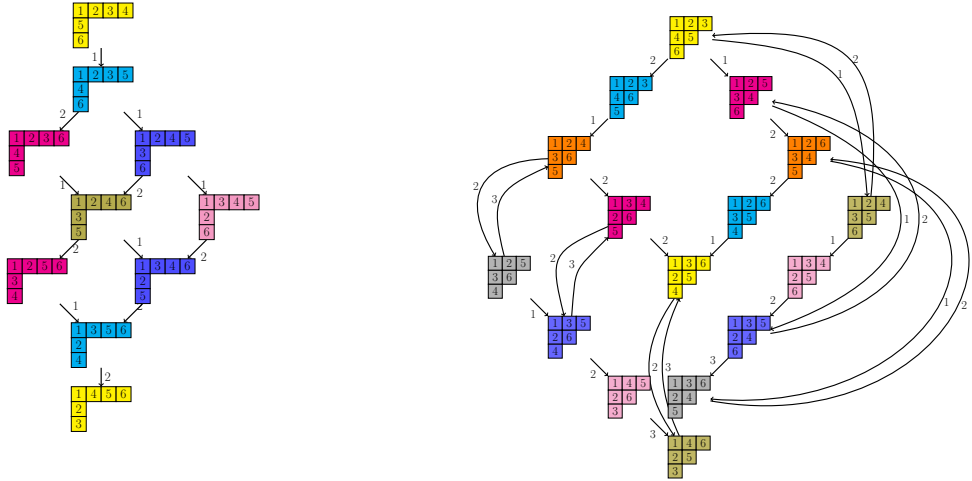

It would be interesting to study further this notion of skeleton of crystals. 
In particular, Danilov, Karzanov and Koshevoy have done so, along with studying the notion of subcrystals \cite{CrossingModel}, in the alternative crossing model for $A_{n-1}$ crystals. They introduced alternate combinatorial objects as vertices of crystals, defined crystal operators on these objects by using feasible functions, and showed that this does give an alternative model for $A_{n-1}$ crystals by using Stembridge axioms. Some results may then have connections to those found here, but the vastly different setting makes comparisons difficult. It would however be extremely interesting to further study the connections with their results.

\section{Crystal skeleton and dual equivalence graphs as relations between the plactic and coplactic monoids}
\label{section:RSKPlacticCoplactic}
\label{section:compLittGraphStdtab}

In this section, we explore the relationship between the skeleton and the dual equivalence graphs introduced by Assaf \cite{AssafDualEquivGraph}, which is another oriented graph structure on standard tableaux. 
We will see how the fundamental quasisymmetric functions can be seen as describing the relationship between the plactic and coplactic monoid, and relations between them encode dual $RSK$ equivalences. Let's start by recalling certain definitions.

\subsection{RSK algorithm, jeu de taquin, plactic and coplactic monoids in crystals}

Recall that the RSK algorithm associates to any word $w$ a pair of tableaux $(P(w),Q(w))$. 
See \cite{Fulton} or \cite{Sagan} for a full description of the algorithm. The tableau $P(w)$ is called the insertion tableau of $w$, and will have the same weight as $w$. The tableau $Q(w)$ is called the recording tableau of $w$, and is standard in this context.\\

Words then form the plactic monoid, 
with concatenation as product, and Knuth relations as equivalence relations \cite{MonoidePlaxique}.
All words in the same equivalence class in the plctic monoid are mapped onto the same insertion tableaux $P$, and its row reading word $rw(P)$ can be seen as a representative of this Knuth-equivalence class. We will take this as the definition for words to be \emph{plactically equivalent}.\\

\emph{Jeu de taquin} 
allows, among other things, to translate Knuth relations (of the plactic monoid on words) to tableaux, and to describe a crystal structure on skew tableaux: tableaux on shapes $\lambda/\mu$, where the cells of $\mu$ are blanks in $\lambda$, and other cells are filled with the usual row and column conditions. Its effect on these skew tableaux then corresponds to applying the Knuth relations to the associated reading words. \\

Starting with a skew tableau, blanks pass through non-empty cells, always preserving conditions on rows and columns. A jeu de taquin slide always starts at an inner corner, having non-empty cells to its right and under it, exchanging it successively with non-empty cells until it lies on the outer shape of $\lambda$ and no more exchanges are possible. Doing this process recursively allows one to "rectify" the tableau to a partition shape. This tableau is called the \emph{rectification} of the initial skew tableau. The rectification is unique, so the order of the slides doesn't matter.  \\

For example, the skew tableau $T$ below, of skew shape $(5,5,3)/(2,2)$, is rectified in five jeu de taquin slides, where the inner corners used for the slides are identified by red cells, and the entries moved in the slide appear in red:
\begin{center}
	$T=$ \hspace{-0.5em}$\gyoung(::!\rd;<\cdot>!\wt!\red;1;1!\black,::;2;2!\red;3!\black,;1;2;3)$ \hspace{-0.5em} $ \rightarrow$ \hspace{-1em} 
	$\gyoung(::;1;1;3,:!\rd;<\cdot>!\wt;2;2,;1!\red;2;3!\black) $ \hspace{-1em} $\rightarrow$ \hspace{-1em}
	$\gyoung(:!\rd;<\cdot>!\wt!\red;1;1!\black;3,:;2;2!\red;2!\black,;1;3) $ \hspace{-1em} $\rightarrow$
	$\gyoung(:;1;1;2;3,!\rd;<\cdot>!\wt;2;2,!\red;1;3!\black)$ \hspace{-1em} $\rightarrow$ \hspace{-0.5em}
	$\gyoung(!\rd;<\cdot>!\wt;1;1;2;3,!\red;1;2;2!\black,;3)$ \hspace{-1em} $ \rightarrow$ \hspace{-0.5em}
	$\gyoung(;1;1;1;2;3,;2;2,;3)$ \hspace{-0.5em}
	$=Rect(T)$
\end{center}

At each step the skew tableau remains semistandard and (skew) partition shaped.
Jeu de taquin slides commute with crystal operators \cite{vLee2001LRrule}, so a crystal of skew tableaux is isomorphic to the crystal where all skew tableaux have been rectified. 
In particular, the crystal of skew tableaux with the parsing of the reading words in maximal increasing factors giving exactly the rows of the tableaux, as in the introduction, is rectified to the crystal with tableaux obtained from those words through the RSK algorithm. \\

The following results give relations between crystals on words and on tableaux:

\begin{prop}[{\cite[Theorems 8.6 and 8.7]{BumpSchilling}}]
	The crystal on words of length $k$ with letters in $[n]$, noted $[n]^{\otimes k}$, 
	decomposes into a disjoint union of crystals, each isomorphic to a certain $B(\lambda)$, for $\lambda$ partitions of $k$ with length at most $n$.\\
	
	Let $x,y\in [n]^{\otimes k}$, then 
	\begin{enumerate}[itemsep=0ex]
		\item $P(f_i(x)) = f_i(P(x))$, where $f_i$ denotes a crystal operator on words or tableaux depending on the object it is applied to.
		\item If $\lambda$ is the shape of $P(x)$ and $Q(x)$, then $x$ lies in a connected component of the crystal isomorphic to $B(\lambda)$.
		\item If $P(x)=P(y)$, then $x,y$ are plactically equivalent.
		\item If $Q(x)=Q(y)$, then $x,y$ lie in the same connected component of the crystal on words.
	\end{enumerate}
	\label{prop:crystalOfTableaux}
\end{prop}

Plactically equivalent words then lie in the same position in isomorphic crystals, and the distinct recording tableaux $Q(w)$ allow us to differentiate between isomorphic components. We can then consider pairs $(P(w),Q(w))$ as the vertices of crystals of tableaux. 
In general, we only consider the first tableau of the pair, since crystal operators act only on them.\\

It is useful to know that $RSK$ preserves descents, and so descent compositions, of words in the recording tableau, as seen in the introduction:

\begin{prop}[{\cite[Theorem 10.117]{loehr2011bijective}}]
	Let $w$ be a word, and $Q(w)$ the associated recording tableau obtained through the RSK algorithm. Then \[Des(Q(w))=Des(w).\]
	\vspace{-1em}
	\label{prop:RSKPreservesDescents}
\end{prop}

All words with the same recording tableau $Q(w)$ are said to be coplactically equivalent (in the associated coplactic monoid). They all land in the same connected crystal. In particular, coplactic equivalences preserve descents by the above proposition.
	
\subsection{Relation between plactic and coplactic monoids: skeleton and dual equivalence graphs}

The relation between plactic and coplactic classes is illustrated by the well known fact that if $RSK(w)=(P(w),Q(w))$, then $RSK(w^{-1})=(Q(w),P(w))$ (see Theorem 10.112 of \cite{loehr2011bijective}). This correspondence works for any word, not only permutations, by using the $RSK$ map on biwords, where the inverse of a biword is easily computed. In this context, the recording tableau $Q(w)$ may be semistandard.\\ 
%
	
This can then be understood by linking standard tableaux of shape $\lambda\vdash m$ indexing connected crystal components $B(\lambda)$ in $[n]^{\otimes m}$, and the associated subcomponents of $B(\lambda)$. 
%
They can be seen as the inverse images of the relation in RSK, exchanging the $P,Q$-tableaux as the effect of considering $w^{-1}$: standard tableaux expand into their associated subcomponent as the subcomponents shrink to their associated standard tableau, respectively under destandardization in all posssible ways which preserve the minimal parsing, and standardization. The fundamental quasisymmetric functions $F_\alpha$ and their associated subcomponents can then be seen as representing the relation between plactic and coplactic classes. This new interpretation of these relationships establishes additional relations between connected components $B(\lambda)$ in the tensor $B(1)^{\otimes m}$.\\

Other interesting relations are often referred to as \textit{dual RSK relations}: equivalences, generally on permutations, with the following elementary transformations on letters $i-1$, $i$ and $i+1$ in a permutation, according to their relative positions:
\begin{center}
$\begin{array}{ccccccc} 
	i & i+1 & i-1 & \xleftrightarrow{i} & i-1 & i+1 & i \\
	i & i-1 & i+1 & \xleftrightarrow{i} & i+1 & i-1 & i \\
\end{array}$
\end{center}

Note that these elementary transformations do not correspond to coplactic relations, in particular since they introduce (or remove) descents.\\

Assaf introduced dual equivalence graphs to represent these dual RSK relations on permutations as actions on the standard tableaux with these permutations as reading words \cite{AssafDualEquivGraph}. They are then graphs on standard tableaux, just like our skeleton of crystals. \\


Even though they are defined very differently, these two types of oriented graph structures on standard tableaux are surprisingly similar. The associated dual graph for $\lambda=(4,3)$, as defined by Assaf, would be the one illustrated in figure~\ref{fig:dualGraphAssaf}. Its graph structure only differs from that of figure~\ref{fig:graphTabStd} by seven edges missing.

\begin{figure}[h!]
	\Yboxdim{0.7cm}
	\begin{center}
		\resizebox{0.5\textwidth}{!}{
			\begin{tikzpicture}[scale=0.6]
			\tikzset{myptr/.style={decoration={markings,mark=at position 1 with {\arrow[scale=3,>=stealth]{>}}},postaction={decorate}}}
			\tikzstyle{every node}=[font=\Large]
			\begin{pgfonlayer}{nodelayer}
			\node [style=none] (0) at (0, 8) {\Yfillcolour{yellow} \young(1234,567)};
			\node [style=none] (1) at (7, 2) {\Yfillcolour{cyan} \young(1236,457)};
			\node [style=none] (2) at (-7, 5) {\Yfillcolour{yellow} \young(1237,456)};
			\node [style=none] (3) at (0, -1) {\Yfillcolour{blue!70}\young(1256,347)};
			\node [style=none] (4) at (14, -1) {\Yfillcolour{olive!70}\young(1235,467)};
			\node [style=none] (5) at (7, -4) {\Yfillcolour{magenta!60}\young(1245,367)};
			\node [style=none] (6) at (-7, -4) {\Yfillcolour{magenta!60}\young(1257,346)};
			\node [style=none] (7) at (-14, -7) {\Yfillcolour{olive!70}\young(1247,356)};
			\node [style=none] (8) at (0, -7) {\Yfillcolour{blue!70}\young(1345,267)};
			\node [style=none] (9) at (-7, -10) {\Yfillcolour{cyan}\young(1347,256)};
			\node [style=none] (10) at (14, -10) {\young(1246,357)};
			\node [style=none] (11) at (7, -13) {\young(1346,257)};
			\node [style=none] (12) at (0, -16) {\young(1356,247)};
			\node [style=none] (13) at (-7, -19) {\young(1357,246)};
			\end{pgfonlayer}
			\begin{pgfonlayer}{edgelayer}
			\draw [<->, ultra thick, in=105, out=-15, looseness=1.25] (0) to node[above]{4} (4);
			\draw [<->, ultra thick] (1) to node[above]{6} (4);
			\draw [<->, ultra thick, in=360, out=90, looseness=1.50] (4) to node[right]{5} (0);
			\draw [<->, ultra thick, bend right=60] (2) to node[left]{3} (7);
			\draw [<->, ultra thick, in=-135, out=100, looseness=1] (7) to node[left]{4} (2);
			\draw [<->, ultra thick, bend left=90, looseness=1.5] (1) to node[right]{3} (10);
			\draw [<->, ultra thick, bend right=90, looseness=1.8] (10) to node[right]{4} (1);
			\draw [<->, ultra thick] (4) to node[above left]{3} (5);
			\draw [<->, ultra thick] (3) to node[above left ]{6} (6);
			\draw [<->, ultra thick, bend right=40, looseness=1.4] (3) to node[right]{2} (12);
			\draw [<->, ultra thick, bend left=30, looseness=1.3] (12) to node[right]{3} (3);
			\draw [<->, ultra thick] (5) to node[above]{2} (8);
			\draw [<->, ultra thick, bend right=10] (5) to node[right]{5} (10);
			\draw [<->, ultra thick, bend right=10] (10) to node[right]{6} (5);
			\draw [<->, ultra thick] (6) to node[above left]{5} (7);
			\draw [<->, ultra thick, in=160, out=180, bend right=100, looseness=2.5] (6) to node[left]{2} (13);
			\draw [<->, ultra thick, in=160, out=180, bend left=100, looseness=2.2] (13) to node[left]{3} (6);
			\draw [<->, ultra thick, bend left] (7) to node[below left]{2} (9);
			\draw [<->, ultra thick, bend right=15] (8) to node[below left]{5} (11);
			\draw [<->, ultra thick, bend right=15] (11) to node[below left]{6} (8);
			\draw [<->, ultra thick, 
			bend right] (9) to node[left]{4} (13);
			\draw [<->, ultra thick, in=-80, out=70, looseness=1,
			bend right] (13) to node[left]{5} (9);
			\draw [<->, ultra thick] (10) to node[above left]{2} (11);
			\draw [<->, ultra thick 
			] (11) to node[above left]{4} (12);
			\draw [<->, ultra thick] (12) to node[above left]{6} (13);
			\end{pgfonlayer}
			\end{tikzpicture}}
	\end{center}
	\caption{Dual equivalence graph for $\lambda=(4,3)$, as of the definition of Assaf \cite{AssafDualEquivGraph}.\label{fig:dualGraphAssaf}}
\end{figure}
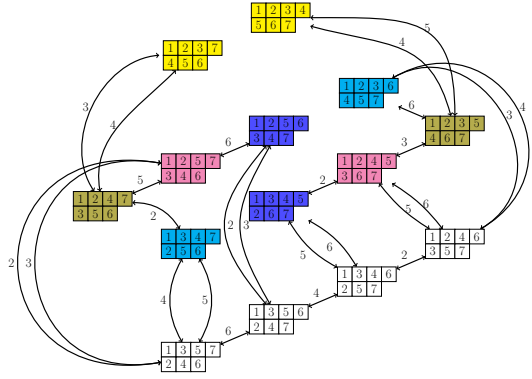

Franco Saliola suggested the dual equivalence graph for $\lambda$ is a subgraph of $Skeleton(\lambda)$. We conjecture further that

\begin{conj}
	The dual equivalence graph for $\lambda$ 
	is a subgraph of $Skeleton(\lambda)$, in the sense that, forgetting orientations and labels, if there are $r$ two sided arrows between two standard tableaux in the dual equivalence graph, there are also $r$ edges between the same standard tableaux in $Skeleton(\lambda)$.\\
	
	Moreover, if $Skeleton(\lambda)$ has $r>1$ edges between two standard tableaux, then there are $r$ edges between the same standard tableaux in the dual equivalence graph for $\lambda$.
	\label{conj:dualEquivGraph}
\end{conj}

This conjecture has been verified for all partitions $\lambda\vdash m$ with $m\leq6$. We also note that most of the time the two graph structures are equal. 
The skeleton 
would then encode the dual RSK equivalences (among other relations) on standard tableaux.\\

Note that the reading word of all tableaux with same minimal parsing standardize to the same permutation, as noted in example~\ref{ex:StdTab}. All these tableaux correspond to the vertices of a subcomponent $B(T_\alpha)$, mapped onto $std(T_\alpha)$ in the skeleton.
We can then see the skeleton as encoding relations between permutations (reading words of the standard tableaux). However, the fact that these relations are generally dual RSK equivalences is surprising.

%

\section{Applications to plethysm}
\label{section:plethysm}

\subsection{Counting monomials in plethysms $s_\mu[s_\lambda]$}

As discussed in the introduction, plethysms of two Schur functions $s_\mu[s_\lambda]$ have a decomposition in the basis of fundamental quasisymmetric functions \cite{LoehrWarrington}: \[s_\mu[s_\lambda] = \sum_{A\in S_{a,b}(\mu,\lambda)} F_{Asc(A)}.\]
In this formula, $\mu\vdash a$, $\lambda\vdash b$, $S_{a,b}(\mu,\lambda)$ is a set of $a\times b$ "standard" matrices which depend on the shapes $\mu$ and $\lambda$, and $Asc(A)$ gives a composition giving the ascents of the word read off the matrix $A$ under a complex reading order.\\

The matrices are built from what Loehr and Warrington call tableaux of tableaux: entries of the tableau of shape $\mu$ are tableaux of shape $\lambda$. We then say they have \emph{shape $\lambda^\mu$}.
%
%
%
When a total order on tableaux is fixed, the tableaux of tableaux of shape $\lambda^\mu$ give the monomials of $s_\mu[s_\lambda]$ (see for example \cite{DeBoeckPagetWildon}, \cite{LoehrWarrington}, \cite{Stanley2}, etc.).\\ 

One definition of plethysm is in terms of the variable substitution of the $x_1,x_2,\ldots$ in $s_\mu$ by the monic monomials of $s_\lambda$ (monomials with coefficient $1$). 
Monomials with coefficients $c\in\N$ greater than $1$ are simply broken down into $c$ monic monomials. Since both functions in the plethysm are Schur functions (so symmetric), then the order of the monomials is not important, and the concept of tableaux of tableaux makes perfect sense.\\

By using the results above, we have that
\begin{cor}
	The number of monic monomials in a plethysm $s_\mu[s_\lambda(x_1,\ldots,x_n)]$ is equal to $|SSYT(\mu)_{|SSYT(\lambda)_n|}|$.
\end{cor}

\begin{proof}
	One need only consider the plethystic substitution of the monic monomials of the Schur function $s_\lambda(x_1,x_2,\ldots,x_n)$ into $s_\mu$.
\end{proof}

\subsection{Decomposing a symmetric sum of quasisymmetric functions into the basis of Schur functions}

Since plethysms are symmetric, and the plethysm of two Schur functions can be expressed as a symmetric sum of fundamental quasisymmetric functions, giving a combinatorial description of the passage from this expression to one in the Schur basis might help make progress on plethysm problems.\\

If $f$ is a symmetric function with decomposition into the basis of quasisymmetric functions $f = \sum d_\alpha F_\alpha$, we can replace the $F_\alpha$ by generalized Schur functions $s_\alpha$, defined using the Jacobi-Trudi definition on determinants, with the same coefficients \cite{Garsia2018}. A generalized symmetric function $s_\alpha$ is equal to $\pm s_\lambda$, for some partition $\lambda$. Then a lot of generalized Schur functions cancel out, so this is far from efficient.\\

For a symmetric $f = \sum d_\alpha F_\alpha$, another way to express it in the Schur basis is through multiple changes of basis: fundamental quasisymmetric functions to monomial quasisymmetric functions to monomial symmetric functions to Schur functions.  This is computationally faster, so this is the algorithm implemented in SageMath. However, this doesn't give much insight into the relationship between the two basis which truely interests us.\\ 

We can then use the results above to give another way of expressing a symmetric function (given in terms of fundamental quasisymmetric functions) into the basis of Schur functions. 

\begin{prop}
	Let $f$ be a symmetric function which admits a decomposition $f=\sum_{\beta} c_\beta F_\beta$ into the basis of fundamental quasisymmetric functions. Let $\alpha$ be the maximal descent composition appearing in the decomposition of $f$ for the lexicographical order.\\ 
	Then $\alpha$ is a partition, $f-c_\alpha s_\alpha$ is Schur-positive and $F$-positive.
\end{prop}

\begin{proof}
	Since $f$ is symmetric, it must admit a (unique) decomposition into the basis of Schur function.
%
	%
	We have seen in proposition~\ref{prop:conditionsOfDescentCompositionsAndPartitions} that the partitions $\lambda$ such that $F_\alpha$ can appear in $s_\lambda$ must have $\lambda_1\geq \alpha_i$, so in particular $\lambda_1\geq \alpha_1$. We must also have that $\lambda_1+\lambda_2+\ldots+\lambda_j\geq \alpha_1+\alpha_2+\ldots+\alpha_j$ for all $j$, so $\alpha\leq_{lex} \lambda$, with $\leq_{lex}$ the lexicographical order.\\
	
	If $s_\lambda$ appears in $f$, then $F_\lambda$ must appear in the decomposition of $f$ in the basis of quasisymmetric functions. This is because $s_\lambda = \sum_{T\in SYT(\lambda)} F_{compDes(T)}$, and $\lambda$ is the descent composition of the standard tableau often refered as the \textit{superstandard tableau}, which has destandardization (according to its minimal parsing) $1_\lambda$.\\

	Since we have picked $\alpha$ maximal for the lexicographical order in all descent compositions appearing in the decompositions of $f$, then $c_\gamma = 0$ for all $\gamma>_{lex}\alpha$. We must then have that $\alpha=\lambda$ is a partition, $s_\alpha$ appears $c_\alpha$ times in the decomposition of $f$ in the basis of Schur functions, and so $f-c_\alpha s_\alpha$ is symmetric and has a decomposition in the basis of fundamental quasisymmetric functions with only positive coefficients.
\end{proof}

\begin{cor}
	The following algorithm gives the decomposition of a 
	symmetric function $f$, expressed in the basis of fundamental quasisymmetric functions, into the Schur basis.
	
	\begin{algo}
		Let $f$ be a 
		symmetric function with an expression in the basis of fundamental quasisymmetric functions $f=\sum_\beta c_\beta F_\beta$. 
		Let $S=f$, and reset $f=0$. 
		\begin{enumerate}[itemsep=0ex]
			\item Let $\alpha$ be the leading support, ie the largest descent composition appearing in $S$ for the lexicographical order. It must be a partition.
			\item Let $\displaystyle S=S-c_\alpha \left( \sum_{T\in SYT(\alpha)} F_{compDes(T)} \right)$ and $f=f+c_\alpha s_\alpha$.
			\item Repeat until $S=0$.
		\end{enumerate}
		Then $f$ is expressed in the basis of Schur functions.
	\end{algo}  
\end{cor}

\begin{rem}
	This algorithm is rather simple and straightforward from the definitions, so others may have used it before. However, it seems to be absent from the literature. Its underlying construction is its most important interest, as it is generally not more efficient than the algorithm implemented in SageMath (when computing the decomposition of a random symmetric sum of quasisymmetric functions of degree $n$ into the Schur basis). 
	Maybe plethysm is even more closely related to fundamental quasisymmetric functions than we thought.
\end{rem}

\begin{rem}
	This algorithm may explain the result of De Boeck, Paget and Wildon, stating that a Schur function $s_\nu$ occurs in a plethysm $s_\mu[s_\lambda]$ with multiplicity given exactly by the number of \textit{maximal plethystic tableaux of weight $\nu$} if and only if $\nu$ is maximal for the lexicographical order for $\mu,\lambda$ \cite{DeBoeckPagetWildon}. Maximal here indicates that no entry $c$ of a tableau-entry can be changed for an entry $c-1$ without breaking the condition of the larger tableau being semistandard.
\end{rem}

\section*{Conclusion}

We now have a better understanding of the decomposition of Schur functions into quasisymmetric functions, and of the relationships between them. 
We may now use this to find a more direct expression of the plethysm of two Schur functions into the basis of Schur functions.

\section*{Aknowledgement} I thank Franco Salioa for his support throughout this project. I also thank those who contributed to SageMath, which helped test examples and generate figures, and those behind OEIS, which helped formulate proposition~\ref{prop:NumberOfTableauxInBm}. FMG received funding from NSERC. 

\newpage
\appendix
\section{Annex: Proofs for $\operatorname{EVAC}$}
\label{annex:proofsEVAC}

We give here a proof that the evacuation map $\operatorname{EVAC}$ defined in section~\ref{section:dualposition} is an anti-automorphism of crystals. We also give a proof that it inverses descent compositions. Berenstein and Zelevinsky proved the former in \cite{BerensteinZelevinsky}, but their proof uses a lot of tools of crystal theory and representation theory. We believe our proof to be of interest, since it can be more accessible, and also uses an anti-automorphism of crystals on words $Rot$ studied already in the litterature: 
for $w=w_1w_2\ldots w_k\in [n]^{\otimes k}$, \[Rot(w)= \text{compl}(w_k\ldots w_2 w_1) = \bar{w_k}\ldots \bar{w_2} \bar{w_1},\] where $\bar{\ell}=compl(\ell)=n-\ell+1$.\\

The operator $Rot$ was studied by Poirier and Reutenauer in \cite{PoirierReutenauer}. They showed that
	$Rot(w) = w_0 w w_0$, where $w_0 = k (k-1) \ldots 3 2 1$ is the longest permutation of the symmetric group $\S_k$, for $k$ the length of $w$. 
%
	$w w_0$ is the mirror image $w$, 
	and 
	$w_0 w$ changes letters $i$ of $\overset{\leftarrow}{w}$ into $k-i+1$. If $w$ is not a permutation, we can standardize $w$ according to its weight $\beta$ from left to right, and afterwards de-standardize it according to $\overset{\leftarrow}{\beta}$, from right to left. We then get the same effect on $w$ as $Rot$.\\

We then have that

\begin{prop}
	$Rot$ is an anti-automorphism of crystals of words which inverses descent compositions.
\end{prop}


\begin{proof}
	We have that $Rot$ is an involution, so we have an automorphism.
	We need to show that $Rot (f_i(w)) = e_{n-i}(Rot(w))$ for any word $w$. \\
	
	By definition, $Rot(w)=\bar{w_k}\bar{w_{k-1}}\ldots \bar{w_2}\bar{w_1}$ where $\bar{w_i}=n-w_i+1$.
	Suppose entries $i$ and $i+1$ give a certain sequence of unpaired parenthesis $)^{\phi_i}  (^{\epsilon_i}$. Then $Rot(w)$ has the sequence $)^{\epsilon_i}  (^{\psi_i}$ for entries $n-i$ and $n-i+1$, where entries $n-i+1$ are obtained from letters $i$ in $w$, and letters $n-i$, from letters $i+1$ in $w$.
	The letter $i$ affected by $f_i$ in $w$ corresponds then to the letter $n-i+1$ affected by $e_{n-i}$ in $Rot(w)$. Then $Rot(f_i(w))=e_{n-i}(Rot(w))$.\\
	
	Now descent compositions have been described as the lengths of the minimal parsing of words $w$ into weakly increasing factors. If a word $w$ has minimal parsing of lengths $\alpha$, then reversing the order of the letters gives weakly increasing sequences of lengths $\overset{\leftarrow}{\alpha}$, 
	and complementing letters gives back weakly increasing sequences of lengths $\overset{\leftarrow}{\alpha}$. Then $Rot$ inverses descent compositions.
\end{proof}

Let's denote by $B_\lambda(w)$ the crystal (on words) which has source $w$, of weight $\lambda$. Then $B_\lambda(w)\isom B(\lambda)_n$ by proposition~\ref{prop:IsomCristalTab}.

\begin{rem}
	It is possible to show that crystal operators preserve descent compositions on words \cite{ApplebyWhitehead}, so the fact that $Rot$ inverses descent compositions of words implies that the image of a connected component crystal of words $B_\lambda(w)$, under the $Rot$ map, is the (dual) isomophic connected component of crystal of words $B_\lambda(w)^\#$, obtained by reversing arrows directions, re-labelling arrows $i$ by $n-i+1$, and vertices $v$ by $Rot(v)$.\\
	
	In particular, if $w$ is a Yamanouchi word (the source of a connectd component $B_\lambda(w)$), then 
	$Rot(w)$ is an anti-Yamanouchi word, 
	with $f_i(Rot(w)) = NULL \ \forall i$, and 
	the sink of $B_\lambda(w)^{\#}$. 
	
	These two components are isomorphic to the same $B(\lambda)$, so the sources and sinks of both will map respectively on $1_\lambda$ and the sink of $B(\lambda)$, which will be $\operatorname{EVAC}(1_\lambda)$.
\end{rem}



We can now use the correspondence given by Poirier and Reutenauer to see how this $Rot$ map affects the $RSK$ insertion. Since $Rot(w) = w_0 w w_0$, then $RSK(Rot(w))=RSK(w_0 w w_0)$.

\begin{prop}[\cite{ChmutovFriedenKimLewisYudovine}]%
	$RSK(w_0 w w_0) = (\operatorname{EVAC}(P(w)), \operatorname{EVAC}(Q(w)).$
\end{prop}

We can finally can show that $\operatorname{EVAC}$ is an anti-automorphism of crystals of tableaux, by using only the results above.
\begin{prop}
	$\operatorname{EVAC}$ is an anti-automorphism of crystals of tableaux, when considering vertices as pairs of tableaux $(P,Q)$ with $Q$ standard obtained through $RSK$.
\end{prop}

\begin{proof}
	First off, lets note that $\operatorname{EVAC}(Q)$ is a standard tableau which indexes an isomorphic connected component of crystals of tableaux. Then pairs $(P,\operatorname{EVAC}(Q))$ are the vertices of this isomorphic connected component, for all $P$ appearing in the first connected component.\\
	
	All tableaux $P$ of shape $\lambda$ appear in $B(\lambda)_n$. Since the involution $\operatorname{EVAC}$ sends a tableau of shape $\lambda$ onto another tableau of the same shape, then it is an auto(iso)morphism on $B(\lambda)_n$.\\
	
	If $w$ is Yamanouchi, it is sent onto $(1_\lambda,Q(w))$ through $RSK$, and $Rot(w)$ is sent onto $(\operatorname{EVAC}(1_\lambda), \operatorname{EVAC}(Q(w)))$. $\operatorname{EVAC}(1_\lambda)$ is then the sink of $B(\lambda)_n$, by the above remark. 
	If $v = f_{i_1}f_{i_2}\ldots f_{i_\ell}(w)$, for $w$ the Yamanouchi word which is the source of the connected component of the crystal of words in which lies $v$, then 
	\begin{align*}
	RSK(Rot(v)) &= RSK(Rot(f_{i_1}f_{i_2}\ldots f_{i_\ell}(w))) \\ &= RSK(e_{n-i_1}e_{n-i_2}\ldots e_{n-i_\ell}(Rot(w)) \\ &= (e_{n-i_1}e_{n-i_2}\ldots e_{n-i_\ell}(\operatorname{EVAC}(1_\lambda)), \operatorname{EVAC}(Q(w))),
	\end{align*}
	
	where the crystal operators are either those on words or on tableaux depending on the object they apply to.
	The fact that $RSK$ commutes with crystal operators, going from second equality to the third, follows from proposition~\ref{prop:crystalOfTableaux}.\\
	
	We then have that $\operatorname{EVAC}(f_i(P)) = e_{n-i}(\operatorname{EVAC}(P))$, and that $\operatorname{EVAC}$ is a crystal anti-automorphism.
\end{proof}

\begin{prop}
	$\operatorname{EVAC}$ inverses descent compositions of tableaux:\\ if $DesComp(T) = \alpha = (\alpha_1, \alpha_2, \ldots, \alpha_s)$, then
	$DesComp(\operatorname{EVAC}(T)) = \overset{\leftarrow}{\alpha} = (\alpha_s, \ldots, \alpha_2,  \alpha_1)$.
\end{prop}

\begin{proof}
	We have seen in section~\ref{section:RSKPlacticCoplactic} that a word $w$ and its recording tableau $Q(w)$ share the same descent composition, and that $RSK(w^{-1})=(P(w^{-1}), Q(w^{-1}))=(Q(w),P(w))$. For a fixed tableau $T$, if $rw(T)=w$, then $w^{-1}$ and $P(w)=T$ share the same descent composition.\\ 
	
	We have also seen that $RSK(Rot(w)) = (\operatorname{EVAC}(P(w)),\operatorname{EVAC}(Q(w)))$, so 
	\begin{align*}
	RSK(Rot(w)^{-1}) &= (\operatorname{EVAC}(Q(w)),\operatorname{EVAC}(P(w)))\\
	&= RSK(Rot(w^{-1})).
	\end{align*}
	
	Then $Rot(w^{-1})$ and $\operatorname{EVAC}(P(w))=\operatorname{EVAC}(T)$ share the same descent composition.\\
	
	Finally, since $Rot$ inverses descent compositions on words, then if the descent composition of $w^{-1}$ and $P(w)=T$ is $\alpha$, then the descent composition of $Rot(w^{-1})$ and $\operatorname{EVAC}(T)$ is $\overset{\leftarrow}{\alpha}$.
\end{proof}

\begin{figure}[h]
	\begin{center}
		\begin{tabular}{ccc}
			\textcolor{blue}{$B_\mu(w)$} & & \textcolor{blue}{$B_\mu(w)^\#$} \vspace{2em} \\
			$w$ & \multirow{3}{3cm}{$\quad \nearrow\mathllap{\searrow} \quad Rot$ } & $w'$ \\
			$\left( \quad \newline \quad \newline \quad \right)$ & & $\left( \quad \newline \quad \newline \quad \right)$\\
			$Rot(w')$ & & $Rot(w)$
			\vspace{1em}\\
			& \textcolor{red}{$ \quad \downarrow RSK \quad \quad $} & 			\vspace{1.5em} \\
			$(1_\mu, \quad \quad \quad \quad \quad Q(w))$ & \multirow{3}{3cm}{$\quad \nearrow\mathllap{\searrow} \quad \operatorname{EVAC}$ } & $(1_\mu, \quad \quad \quad \quad \quad \operatorname{EVAC}(Q(w)))$ \\
			\newline
			$\left( \quad \newline \quad \newline \quad \right) $  & & $\left( \quad \newline \quad \newline \quad \right) $ \\
			\newline
			$(\operatorname{EVAC}(1_\mu), \quad Q(w))$ & & $(\operatorname{EVAC}(1_\mu), \quad  \operatorname{EVAC}(Q(w)))$\\
		\end{tabular}
	\end{center}
	\caption{Bijection between connected components isomorphic to $B(\mu)$ which are linked by the anti-automorphisms $Rot$ and $\operatorname{EVAC}$, and the isomorphism of crystals $RSK$, with $\operatorname{EVAC}(Q(w)) = Q(w')$ and $w'=RSK^{-1}(1_\mu,\operatorname{EVAC}(Q(w)))$.}
\end{figure}


\newpage
\bibliographystyle{apalike}
\bibliography{QSym.bib}

\end{document}